\documentclass{interact}
\usepackage{amsfonts}
\usepackage{amssymb}
\usepackage{amsmath}
\usepackage{comment}
\usepackage{todonotes}
\usepackage[english]{babel}
\usepackage{secdot}
\usepackage{hyperref}
\usepackage{url}
\usepackage{mathtools}
\usepackage{algorithm}
\usepackage{algorithmic}

\usepackage{comment}
\usepackage{subfig}
\newcommand{\eps}{\varepsilon}

\newcommand{\la}{\langle}
\newcommand{\ra}{\rangle}
\def\vp{\varphi}
\def\R{\mathbb{R}}
\DeclarePairedDelimiter{\ceil}{\lceil}{\rceil}

\def\pd#1{{\color{red}#1}} 

\newtheorem{theorem}{Theorem}
\newtheorem{lemma}{Lemma}
\newtheorem{remark}[]{Remark}
\begin{document}
\title{Primal-dual accelerated gradient methods with small-dimensional relaxation oracle}

\author{\name{Yurii Nesterov\textsuperscript{a} \and Alexander Gasnikov\textsuperscript{b,c} \and Sergey Guminov\textsuperscript{b,c} \and Pavel Dvurechensky\textsuperscript{d,c}} \affil{\textsuperscript{a} Center for Operations Research and Econometrics (CORE), Catholic University of Louvain (UCL), Louvain-la-Neuve, Belgium;
National Research University Higher School of Economics, Moscow, Russia;
\textsuperscript{b}
Moscow Institute of Physics and Technology, Dolgoprudny, Russia;
\textsuperscript{c}
Institute for Information Transmission Problems RAS, Moscow, Russia;
\textsuperscript{d}
Weierstrass Institute for Applied Analysis and Stochastics, Berlin, Germany
}}

\date{Received:  / Accepted: }

\maketitle

\begin{abstract}
In this paper, a new variant of accelerated gradient descent is proposed. The proposed method does not require any information about the objective function, uses exact line search for the practical accelerations of convergence, converges according to the well-known lower bounds for both convex and non-convex objective functions, possesses primal-dual properties and can be applied in the non-euclidian set-up. As far as we know this is the first such method possessing all of the above properties at the same time. We also present a universal version of the method which is applicable to non-smooth problems. We demonstrate how in practice one can efficiently use the combination of line-search and primal-duality by considering a convex optimization problem with a simple structure (for example, linearly constrained). 

\keywords{accelerated gradient descent, line-search, primal-dual methods, convex optimization, nonconvex optimization}
\begin{amscode}
90C25, 68Q25
\end{amscode}
\end{abstract}

\section{Introduction}
The first accelerated gradient method for smooth convex optimization problems dates back to 1980s \cite{nesterov1983method}. This method has optimal \cite{nemirovskii1983problem} convergence rate $f(x^k)-f(x_*) = O(1/k^2)$, where $k$ is the iteration counter, $f$ is the objective function and $x_*$ is an optimal point. It is less known that that there were earlier versions of optimal methods. The key difference is that, on each step, those versions used small-dimensional relaxation oracle (sDR-oracle), i.e. auxiliary minimization over some small-dimensional subspace. Thus, these early methods are optimal under additional assumption of availability of the sDR-oracle. This assumption is rarely satisfied in practice, since solving auxiliary minimization problem on each step can be very costly or can lead to divergence of the method due to accumulating error in practice. This seems to be the main reason why mostly the fixed-step accelerated methods \cite{nesterov1983method} have been developing in the last decades \cite{nesterov2004introductory,nesterov2005smooth,nesterov2013gradient}. This choice of the direction of research by the community is supported by the fact that availability of sDR-oracle does not improve the worst-case theoretical guarantee of the optimal gradient-type procedure for smooth convex optimization problems \cite{nesterov1989effective}.

There has been a resurging interest in first-order methods using sDR-oracle \cite{narkiss2005sequential,deklerk2017worst-case,guminov2017universal}.
One of the reasons is that, in practice, small-dimensional relaxation allows local adaptation to the curvature of the objective function, which can dramatically improve the practical convergence rate, the classical example being conjugate gradient methods.
At the same time, there are problem classes for which the computational overhead of the sDR-oracle is minimal \cite{narkiss2005sequential}. This includes popular machine learning problems known as generalized linear models, including SVM, linear regression, logistic regression, etc., and optimization problems with linear equality constraints. In the first case, the problem itself is of the form
\begin{equation}
\label{init}
f(x)=F(A^Tx) \to \min_{x\in \R^n},
\end{equation}
where $A\in \R^{n\times m}$. In the latter case, the dual problem has the same form as \eqref{init}. Even if the matrix $A$ is large and dense, and $F(y)$ can be calculated in $O(n)$ arithmetic operations, the whole small-dimensional relaxation has almost the same complexity as a single calculation of the gradient of $f(Ax)$. Hence, in this case, the operation complexities of the fixed-step methods and the methods with small-dimensional relaxation are almost identical. 

In this paper, we propose a new accelerated gradient method utilizing a sDR-oracle. Instead of using pre-defined sequences of parameters of the method, such as step sizes, we choose these parameters by minimizing the objective in some specially constructed direction. This makes our method related to conjugate gradient methods. The difference is that our method has  $O(1/k^2)$ convergence rate for general convex objectives. Moreover, our method is adaptive to the smoothness of the objective and does not require the Lipschitz constant of the gradient to be known, unlike classic accelerated gradient descent method \cite{nesterov2004introduction}. The method also converges to a stationary point for non-convex objectives, which means that it is capable of adapting to the local convexity of the objective.  We also estimate the rate of convergence of the method in terms of the norm of the gradient for convex, $\gamma$-weakly-quasi-convex and non-convex objectives.

Further on, we analyze potentially non-smooth functions with H\"older-continuous subgradient and propose a generalization of our method for this case. We prove that our method is universal in the sense of \cite{nesterov2015universal}, i.e. does not require any a priori knowledge of the smoothness of the objective and automatically according to the lower bounds for the class of convex objectives with H\"older-continuous subgradient.

Special attention is devoted to the analysis  of the primal-dual properties of the proposed methods for the class of strongly convex linearly constrained problems. In this case the dual problem has the form \eqref{init} and is solved by our method with the ultimate goal to reconstruct the solution to the primal problem.

Finally, we describe how these methods can be accelerated to have optimal linear convergence under the additional assumption that the objective is strongly convex with known parameter of strong convexity.


\textbf{Related work.}
It is quite hard to cover all the vast literature on accelerated gradient methods \cite{nesterov1983method,nesterov2005smooth,beck2009fast,nesterov2013gradient,nesterov2004introduction}. The versions adaptive to the Lipschitz constant of the gradient may be found in \cite{nesterov2013gradient,beck2009fast}. Usually this type of adaptivity is called "line-search" or "backtracking". Papers \cite{beck2014fast,dvurechensky2017adaptive,dvurechensky2018computational} consider the question of primal-duality of these methods in combination with "backtracking" used for adaptivity to the Lipschitz constant. The authors of \cite{ghadimi2016accelerated}, by a special a priori choice of step sizes, construct a single method which works optimally for convex and non-convex problems.
Universal gradient methods for convex problems were proposed in \cite{nesterov2015universal} and extended in \cite{lan2015generalized} for non-convex problems and in \cite{yurtsever2015universal} for primal-dual setting. Finally, there were some attempts to combine universality with small-dimension relaxation \cite{drori2018efficient,guminov2017universal}.
Concerning conjugate gradient methods, we refer the reader to a good survey \cite{neculai200840conjugate} and the classical book \cite{nocedal2006numerical}.



The structure of the paper is as follows. In section 2 we describe the main algorithm and analyze its convergence. We also show that this method admits a stopping criterion. The subsections of section 2 are dedicated to the modifications of the method applicable to $\gamma$-weakly-quasi-convex and strongly convex objectives. Then the dependence of the method on line-search accuracy is discussed briefly. In section 3 we present the universal version of the algorithm and establish its convergence for convex and non-convex objectives. Once again, a subsection is dedicated to strongly convex objectives. Section 4 contains the results concerning the primal-dual properties of our method. Finally, in section 5 one can find the results of numerical experiments. 

\textbf{Notation.} 
Let $E$ be a finite-dimensional real vector space and $E^*$ be its dual. We denote the value of a linear function $g \in E^*$ at $x\in E$ by $\la g, x \ra$. Let $\|\cdot\|$ be some norm on $E$, $\|\cdot\|_{*}$ be its dual, defined by $\|g\|_{*} = \max\limits_{x} \big\{ \la g, x \ra, \| x \| \leqslant 1 \big\}$. Given a vector $g \in E^*$, we denote by $(g)^{\#} = \arg \max_{\|s\|\leqslant 1 } \la g, s \ra$. We use $\nabla f(x)$ to denote any subgradient of a function $f$ at a point $x \in {\rm dom} f$.

We choose a {\it prox-function} $d(x)$, which is continuous, convex on $Q$ and
\begin{enumerate}
	\item admits a continuous in $x \in Q^0$ selection of subgradients 	$\nabla d(x)$, where $Q^0 \subseteq Q$  is the set of all $x$, where $\nabla d(x)$ exists;
	\item $d(x)$ is $1$-strongly convex on $Q$ with respect to $\|\cdot\|$, i.e., for any $x \in Q^0, y \in Q$ $d(y)-d(x) -\la \nabla d(x) ,y-x \ra \geqslant \frac12\|y-x\|^2$.
\end{enumerate}
Without loss of generality, we assume that $\min\limits_{x\in Q} d(x) = 0$.

We define also the corresponding {\it Bregman divergence} $V(x,z) = d(x) - d(z) - \la \nabla d(z), x - z \ra$, $x \in Q, z \in Q^0$. Standard proximal setups, i.e. Euclidean, entropy, $\ell_1/\ell_2$, simplex, nuclear norm, spectahedron can be found in \cite{ben-tal2015lectures}.

\section{Adaptive methods for smooth optimization}
\label{S:smooth}
We consider the optimization problem 
\begin{equation}
\label{COnvProb}
f(x) \to \min_{x \in E},
\end{equation}
and denote a solution to this problem as $x_\ast$. Our main assumption in this section is that the objective $f$ is $L$-smooth, i.e. is continuously differentiable and has Lipschitz-continuous gradient
\begin{equation}
\label{L}
\|\nabla f(y) - \nabla f(x)\| \leqslant L \|x-y\|_*, \quad \forall x, y \in E.
\end{equation}

Our main algorithm in this section is listed as Algorithm \ref{AGMsDR}.
\begin{algorithm}[!h]
\caption{Accelerated Gradient Method with Small-Dimensional Relaxation (AGMsDR)}
\label{AGMsDR}
\begin{algorithmic}[1]
\ENSURE $x^k$
\STATE Set $k = 0$, $A_0=0$, $x^0 = v^0$, $\psi_0(x) = V(x,x^0)$
\FOR{$k \geqslant 0$}
\STATE \begin{equation}
    \label{eq:beta_k_y_k_def}
    \beta_k = \arg\min_{\beta \in \left[0, 1 \right]} f\left(v^k + \beta (x^k - v^k)\right), \quad y^k = v^k + \beta_k (x^k - v^k).
\end{equation}
\STATE 
Option a), $L$ is known, 
\begin{equation}
\label{eq:xkp1_opt_a}
    x^{k+1} = \arg\min_{x \in E} \left\{ f(y^k) + \langle \nabla f(y^k), x - y^k \rangle + \frac{L}{2} \|x - y^k \|^2 \right\}.
\end{equation}
Find $a_{k+1}$ from equation $\frac{a_{k+1}^2}{A_{k} + a_{k+1}} = \frac{1}{L}$.  \\
Option b), 
\begin{equation}
\label{eq:xkp1_opt_b}
h_{k+1} = \arg\min_{h \geqslant 0} f\left(y^k - h(\nabla f(y^k))^{\#}\right), \quad x^{k+1} = y^  k- h_{k+1}(\nabla f(y^k))^{\#}.
\end{equation}
Find $a_{k+1}$ from equation $f(y^k) - \frac{a_{k+1}^2}{2(A_{k} + a_{k+1})} \|\nabla f(y^k) \|_*^2 = f(x^{k+1})$.
\STATE  Set $A_{k+1} = A_{k} + a_{k+1}$.  
\STATE  Set $\psi_{k+1}(x) =  \psi_{k}(x) + a_{k+1}\{f(y^k) + \langle \nabla f(y^k), x - y^k \rangle\}$.
\STATE $v^{k+1} = \arg\min_{x \in E} \psi_{k+1}(x)$
\STATE $k = k + 1$
\ENDFOR
\end{algorithmic}
\end{algorithm}

Here and for all the methods described further we assume that if the equation for $a_{k+1}$ in step 4 admits multiple solutions, then the greater one is chosen.

Before we move to the theoretical results of this section, let us make some remarks. The main new element of the proposed method is in line 3. Unlike known methods \cite{nesterov2004introductory,nesterov2005smooth,allen2014linear}, which use fixed $\beta_k = \frac{k}{k+2}$, we use minimization over the interval $\beta \in [0,1]$. The choice of the fixed stepsize is motivated by the theoretical convergence analysis. Our goal is to choose best possible stepsize with the same convergence rate guarantees.   
Most of the results described further remain the same if the search over the unit interval $[0,1]$ in line 3 is changed to line-search over any subset of $\mathbb{R}$ containing said interval, for instance, the whole real line $\mathbb{R}$. 

Theoretical analysis of Algorithm \ref{AGMsDR} is based on the following theorem. We underline that the convexity of the objective $f$ is not required.
\begin{theorem} 
\label{Main}
After $k$ steps of Algorithm \ref{AGMsDR} for problem \eqref{COnvProb} it holds that
\begin{equation}
    \label{eq:main_recurrence}
    A_{k}f(x^{k}) \leqslant \min_{x \in \mathbb{R}^n} \psi_{k}(x) = \psi_{k}(v^{k}).
\end{equation}
Moreover, $A_k \geqslant\frac{k^2}{4L}$.
\end{theorem}
\begin{proof} 
Denote
$$l_k(x) = \sum_{i=0}^k a_{i+1}\{f(y^i) + \langle \nabla f(y^i), x - y^i \rangle\}.$$
Then 
$$\psi_{k+1}(x) = l_{k}(x) + \psi_0(x) = \psi_{k}(x) + a_{k+1}\{f(y^k) + \langle \nabla f(y^k), x - y^k \rangle\}.$$
First, we prove inequality \eqref{eq:main_recurrence} by induction over $k$. For $k=0$, the inequality holds. Assume that 
$$A_{k}f(x^{k}) \leqslant \min_{x \in \mathbb{R}^n} \psi_{k}(x) = \psi_{k}(v^{k}).$$
Then 
$$\psi_{k+1}(v^{k+1}) = \min_{x \in \mathbb{R}^n} \left\{ \psi_{k}(x) + a_{k+1}\{f(y^k) + \langle \nabla f(y^k), x - y^k \rangle\} \right\}  $$
$$  \geqslant \min_{x \in \mathbb{R}^n} \left\{ \psi_{k}(v^k) +\frac{1}{2}\|x-v^k\|^2 + a_{k+1}\{f(y^k) + \langle \nabla f(y^k), x - y^k \rangle\} \right\} $$
$$\geqslant \min_{x \in \mathbb{R}^n} \left\{ A_{k}f(x^k) +\frac{1}{2}\|x-v^k\|^2 + a_{k+1}\{f(y^k) + \langle \nabla f(y^k), x - y^k \rangle\} \right\}.$$
Here we used that $\psi_{k}$ is a strongly convex function with minimum at $v^k$.

By the definition of $\beta_k$ and $y_k$ in \eqref{eq:beta_k_y_k_def}, we have 
$f(y^k) \leqslant f(x^k)$. By the optimality conditions in \eqref{eq:beta_k_y_k_def}, either
\begin{enumerate}
    \item $\beta_k = 0$, $\langle \nabla f(y^k),x^k - v^k \rangle \geqslant 0$, $y^k = v^k$;
    \item $\beta_k \in (0,1)$ and $\langle \nabla f(y^k),x^k - v^k \rangle = 0$, $y^k = v^k + \beta_k (x^k - v^k)$;
    \item $\beta_k = 1$ and $\langle \nabla f(y^k),x^k - v^k \rangle \leqslant 0$, $y^k = x^k$ .
\end{enumerate}
In all three cases,  $\langle \nabla f(y^k), v^k - y^k \rangle \geqslant 0$. Thus,
$$\psi_{k+1}(v^{k+1}) \geqslant \min_{x \in \mathbb{R}^n} \left\{ A_{k+1}f(y^k) + a_{k+1}\langle \nabla f(y^k), x - v^k \rangle +\frac{1}{2}\|x-v^k\|^2  \right\}  $$
$$\geqslant A_{k+1}f(y^k) - \frac{a_{k+1}^2}{2}\|\nabla f(y^k)\|_*^2,$$
where we used that for any $g \in E^*$, $s \in E$, $\zeta \geqslant 0$, it holds that $\la g,s \ra + \frac{\zeta}{2}\|s\|^2 \geqslant - \frac{1}{2\zeta}\|g\|_*^2$. 
Our next goal is to show that 
\begin{equation} 
\label{eq:ind_fin_step}
A_{k+1}f(y^k) - \frac{a_{k+1}^2}{2}\|\nabla f(y^k)\|_*^2 \geqslant A_{k+1}f(x^{k+1}),
\end{equation}
which proves the induction step.

For option a) in the step 4, using the $L$-smoothness of $f$ and minimizing the r.h.s. of \eqref{eq:xkp1_opt_a}, we have 
$$
f(x^{k+1}) \leqslant f(y^k) + \langle \nabla f(y^k), x^{k+1} - y^k \rangle + \frac{L}{2} \|x^{k+1} - y^k \|^2 = \min_{x \in E} \left(f(y^k) + \langle \nabla f(y^k), x - y^k \rangle + \frac{L}{2} \|x - y^k \|^2 \right)
$$
$$
= f(y^k) - \frac{1}{2L} \|\nabla f(y^k)\|_*^2.
$$
Since, for this option, $\frac{a_{k+1}^2}{A_{k+1}} = \frac{1}{L}$, inequality \eqref{eq:ind_fin_step} holds.
For option b) in the step 4, \eqref{eq:ind_fin_step} holds by the choice of $a_{k+1}$ from the equation
\begin{equation}
    \label{eq:Th:Main_proof_1}
    f(y^k) - \frac{a_{k+1}^2}{2A_{k+1}} \|\nabla f(y^k) \|_*^2 = f(x^{k+1}).
\end{equation}
It remains to show that this equation has a solution $a_{k+1} > 0$.
By the $L$-smoothness of $f$, we have 
\begin{align}
    \label{eq:Th:Main_proof_2}
f(x^{k+1}) &= \min_{h \geqslant 0} f\left(y^k - h(\nabla f(y^k))^{\#}\right) \leqslant \min_{h \geqslant 0} \left( f(y^k) -h \langle \nabla f(y^k), (\nabla f(y^k))^{\#} \rangle + \frac{Lh^2}{2} \|(\nabla f(y^k))^{\#} \|^2 \right) \notag \\
& = f(y^k) - \frac{1}{2L} \|\nabla f(y^k)\|_*^2,
\end{align}
where we used that $\langle \nabla f(y^k), (\nabla f(y^k))^{\#} \rangle = \|\nabla f(y^k) \|_*^2$ and $\|(\nabla f(y^k))^{\#} \|^2 = 1$ by definition of the vector $(\nabla f(y^k))^{\#}$. Since $A_{k+1}=A_k + a_{k+1}$, we can rewrite the equation \eqref{eq:Th:Main_proof_1} as 
\[
\frac{a_{k+1}^2}{2} \|\nabla f(y^k) \|_*^2 + a_{k+1} (f(x^{k+1}) - f(y^k)) +  A_k(f(x^{k+1}) - f(y^k)) = 0.
\]
Since, by \eqref{eq:Th:Main_proof_2}, $f(x^{k+1}) - f(y^k) < 0$ (otherwise $\|\nabla f(y^k)\|_* = 0$ and $y_k$ is a solution to the problem \eqref{COnvProb}), 
\[
a_{k+1}= \frac{f(y^k) - f(x^{k+1}) + \sqrt{(f(y^k) - f(x^{k+1}))^2-2A_k(f(x^{k+1})-f(y^k))\|\nabla f(y^k) \|_*^2}}{\|\nabla f(y^k) \|_*^2} > 0.
\]

Let us estimate the rate of the growth for $A_k$. If in the step 4 optiona a) is used, $\frac{a_{k+1}^2}{A_{k+1}} = \frac{1}{L}$. For the option b), using \eqref{eq:Th:Main_proof_1} and \eqref{eq:Th:Main_proof_2}, we have $\frac{a_{k+1}^2}{A_{k+1}} \geqslant \frac{1}{L}$. Thus, for both options, $\frac{a_{k+1}^2}{A_{k}+a_{k+1}} = \frac{a_{k+1}^2}{A_{k+1}} \geqslant \frac{1}{L}$.
Since $A_1 = a_1 \geqslant \frac{1}{L}$, we prove by induction that $\alpha_k \geqslant \frac{k}{2L}$ and $A_{k} \geqslant \frac{(k+1)^2}{4L} \geqslant \frac{k^2}{4L}$

Indeed,
	\begin{align}
	\alpha_{k+1} & \geqslant \frac{1 + \sqrt{1 + 4A_k L}}{2L} = \frac{1}{2L} + \sqrt{\frac{1}{4L^2} + \frac{A_{k}}{L}} \geqslant 
	\frac{1}{2L} + \sqrt{\frac{A_{k}}{L}} \notag \\
	&\geqslant 
	\frac{1}{2L} + \frac{1}{\sqrt{L}}\frac{k+1}{2\sqrt{L}} =
	\frac{k+2}{2L}. \notag
	\end{align}
	Hence,
	\begin{equation*}
	A_{k+1} = A_k + \alpha_{k+1} \geqslant \frac{(k+1)^2}{4L} + \frac{k+2}{2L} \geqslant \frac{(k+2)^2}{4L}.
	\end{equation*}
\end{proof}

Next result is simple and standard for gradient methods, but we provide it for the sake of completeness of the paper.

\begin{theorem}
\label{Th:GradDecayNonConv}
Let function $f$ be $L$-smooth and Algorithm \ref{AGMsDR} be run for $N$ steps. Then
$$
\min_{k=0,...,N}\| \nabla f(y^k)\|_*^2 \leqslant \frac{2L(f(x^0) - f(x_*))}{N}.
$$
\end{theorem}
\begin{proof}
We have that \begin{equation}
f(x^{k+1})\leqslant f(y^k)-\frac{1}{2L}\|\nabla f(y^k)\|^2_*\leqslant f(x^k)-\frac{1}{2L}\|\nabla f(y^k)\|^2_*. \label{grad_descent_guarantee}
\end{equation}

Summing this up for $k=0,\ldots, N$, we obtain 

\[f(x^0)-f(x_*)\geqslant f(x^{0})-f(x^{N+1})\geqslant \frac{N}{2L}\min_{k=0,\ldots, N}\|\nabla f(y^k)\|^2_2.\]
Consequently, we may guarantee \[\min_{k=0,\ldots, N}\|\nabla f(y^k)\|^2_2\leqslant \frac{2L(f(x^0)-f(x_*))}{N}.\]
\end{proof}


Before we move to the main results, we define $\gamma$-weakly-quasi-convex functions, which are unimodal, but generally non-convex.
We say that $f(x)$ is $\gamma$-weakly-quasi-convex with $\gamma \in (0,1]$ if for all $x \in \mathbb{R}^n$ 
$$\gamma (f(x) - f(x_*))\leqslant \langle \nabla f(x), x - x_* \rangle.$$
Note that convex functions are $1$-weakly-quasi-convex. The converse is generally not true.

\begin{lemma}
\label{WQC_conv}
Let function $f$ be $\gamma$-weakly-quasi-convex and Algorithm \ref{AGMsDR} be run for $N$ steps. Then 
$$A_{k}(f(x^{k}) - f(x_*)) \leqslant (1- \gamma)A_{k}(f(x^0) - f(x_*))+V(x_\ast,x^0).$$
\end{lemma}
\begin{proof}
According to the definition of $\gamma$-weak-quasi-convexity 
$$l_k(x_*) = \sum_{i=0}^k a_{i+1}\{f(y^i) + \langle \nabla f(y^i), x_* - y^i \rangle\}\leqslant
$$
$$
\leqslant \sum_{i=0}^k a_{i+1}\{(1 - \gamma)f(y^i) + \gamma f(x_*)\}.$$
By \eqref{grad_descent_guarantee} and \eqref{eq:beta_k_y_k_def} we have $f(y^i)\leqslant f(x^i)\leqslant f(x^0)$, so

 \[l_k(x_*) \leqslant \sum_{i=0}^k a_{i+1}\{(1 - \gamma)f(x^0) + \gamma f(x_*)\}.\]
From this inequality and \textbf{Theorem} \ref{Main} we have
$$A_{k} f(x_{k}) \leqslant \min_{x \in \mathbb{R}^n} \psi_{k}(x) \leqslant \psi_{k}(x_*) = l_{k-1}(x_*) + V(x_*,x^0) \leqslant
$$
$$
\leqslant \sum_{i=0}^{k-1} a_{i+1}\{(1 - \gamma)f(x^0) + \gamma f(x_*)\} + V(x_*,x^0).$$
From here, since $A_{k} = \sum\limits_{i=0}^{k-1} a_{i+1}$, by rearanging the terms we obtain the statement of the theorem:
$$A_{k}(f(x^{k}) - f(x_*)) \leqslant (1- \gamma)A_{k}(f(x^0) - f(x_*))+ V(x_*,x^0).$$
\end{proof}

\begin{theorem}
\label{Th:RateConv}
Let function $f$ be $1$-weakly-quasi-convex and $L$-smooth and Algorithm \ref{AGMsDR} be run for $N$ steps. Then 
\begin{equation}\label{grad}
\min_{k=[N/2],...,N}\| \nabla f(y^k)\|_*^2 \leqslant \frac{64L^2V(x_*,x^0)}{N^{3}},
\end{equation}
$$
f(x^N) - f(x_*) \leqslant \frac{4LV(x_*,x^0)}{N^2}.
$$

\end{theorem}
\begin{proof}

Applying \textbf{Lemma~}\ref{WQC_conv} with $\gamma=1$, we get

\[f(x^{N}) - f(x_*)\leqslant \frac{V(x_\ast,x^0)}{A_{N}}.\]
Using the lower bound on $A_N$ established in \textbf{Theorem~}\ref{Main}  we obtain that for a convex (or 1-weakly-quasi-convex) objective \[f(x^N) - f(x_*)\leqslant \frac{4LV(x_*,x^0)}{N^2}.\]

Summing up \eqref{grad_descent_guarantee} for $k=\ceil{N/2},\ldots, N$, we obtain

\[f(x^{\left[N/2\right]})-f(x_*)\geqslant f(x^{\left[N/2\right]})-f(x^{N+1})\geqslant \sum\limits_{k=\left[N/2\right]}^N\frac{\|\nabla f(y^k)\|^2_2}{2L}\geqslant \left[N/2\right]\min_{k=\left[N/2\right],\ldots, N}\frac{\|\nabla f(y^k)\|^2_2}{2L}.\]

Finally, we have \[\min_{k=\left[N/2\right],\ldots, N}\|\nabla f(y^k)\|^2_2\leqslant \frac{4L}{N}(f(x^{\left[N/2\right]})-f(x_*))\leqslant \frac{64L^2V(x_*,x^0)}{N^{3}}.\]

\end{proof}

\begin{remark}
Recently in \cite{kim2018generalizing} a special variant of accelerated gradient descent that converges at the rate
\begin{equation}
\label{gradBest}
\|\nabla f(x^N)\|_2^2 = O\left(\frac{L(f(x^0) - f(x_*))}{N^2}\right).
\end{equation} was proposed.

This result seems to be weaker than \eqref{grad}, but actually from  \eqref{gradBest} one can obtain a much stronger result. Indeed, one can perform $N$  iterations of common fast gradient descent and obtain 
$$f(x^N) - f(x_*) = O\left(\frac{LR^2}{N^2}\right).$$
Then one can put $x^0:= x^N$ and perform $N$ iterations of the method from \cite{kim2018generalizing}. Totally, we obtain 
$$\|\nabla f(x^N)\|_2^2 = O\left(\frac{L^2R^2}{N^4}\right).$$
This bound and the bound \eqref{gradBest} are unimprovable, see \cite{nemirovsky1992information,nesterov2012make}.
\end{remark}

\begin{remark}
Note that our method does not require the knowledge about the convexity of the objective function and automatically works either with rate given by Theorem \ref{Th:GradDecayNonConv} or by Theorem \ref{Th:RateConv}.
\end{remark}

\subsubsection{Online stopping criterion}

If the objective is smooth and convex, this method admits an efficient stopping criterion.

By rewriting the statement of \textbf{Theorem~\ref{Main}} we see that \[f(x^k)\leqslant\frac{1}{A_k}\psi_k(v^k)=\frac{1}{A_k}\min_{x\in\mathbb{R}^n}\left[\frac{1}{2}\|x^0-x\|^2_2+\sum_{i=0}^{k-1} a_{i+1}\left\{f(y^i) + \langle \nabla f(y^i), x - y^i \rangle\right\}\right]\]

As it was before, \[l^{k-1}(x)=\sum_{i=0}^{k-1} a_{i+1}\left\{f(y^i) + \langle \nabla f(y^i), x - y^i \rangle+\frac{\mu}{2}\|x-y^i\|^2_2\right\}\]Denote $R=\|x^0-x_*\|_2$   and \[\hat{f}^k=\min_{x:\ \|x-x_0\|\leqslant R} \frac{1}{A_k}l^{k-1}(x).\]

The constraint may be rewritten equivalently  as $\frac{1}{2}\|x-x_0\|^2\leqslant \frac{R^2}{2}$. By strong duality we see that \begin{align*}
\hat{f}^k&=\min_{x\in \mathbb{R}^n}\max_{\lambda\geqslant 0}\left\lbrace\frac{1}{A_k} l^{k-1}(x)+\lambda\left(\frac{1}{2}\|x_0-x\|^2-\frac{R^2}{2}\right)\right\rbrace\\&=\max_{\lambda\geqslant 0} \min_{x\in \mathbb{R}^n}\left\lbrace\frac{1}{A_k} l^{k-1}(x)+\lambda\left(\frac{1}{2}\|x_0-x\|^2-\frac{R^2}{2}\right)\right\rbrace.
\end{align*}

Now we set $\lambda=\frac{1}{A_k}$ and obtain \[\hat{f}^k\geqslant\frac{1}{A_k}\psi_k(v^k)-\frac{R^2}{2A_k}. \]

Then \[f(x^k)-\hat{f}^k\leqslant\frac{R^2}{2A_k}.\] But by convexity of $f(x)$ we have that $\forall k$ $\frac{1}{A_k}l^{k-1}(x)\leqslant f(x)$, which implies that $\hat{f}^k\leqslant f(x_*).$  Finally, we have

\[f(x^k)-f(x_*)\leqslant f(x^k)-\hat{f}^k\leqslant\frac{R^2}{2A_k},\] so the condition $f(x^k)-\hat{f}^k\leqslant \eps$ is an efficient stopping criterion for the AGMsDR method.

\subsection{$\gamma$-weakly-quasi-convex objectives}
Next we describe a method for more general class of $\gamma$-weakly-quasi-convex functions. Algorithm \ref{AGMsDR-WQC} is obtained from Algorithm \ref{AGMsDR} by applying a restart technique.

\begin{algorithm}[!h]
\caption{Accelerated Gradient Method with Small-Dimensional Relaxation (AGMsDR)}
\label{AGMsDR-WQC}
\begin{algorithmic}[1]
\ENSURE $x^k_i$
\FOR{$i \geqslant 0$}
\STATE Set $k = 0$, $A_0=0$, $x_i^0 = v_i^0$, $\psi^i_0(x) = V(x,x_i^0)$
\FOR{$k \geqslant 0$}
\STATE \begin{equation}
    \label{eq:beta_k_y_k_def_wqc}
    \beta_k = \arg\min_{\beta \in \left[0, 1 \right]} f\left(v_i^k + \beta (x_i^k - v_i^k)\right), \quad y_i^k = v_i^k + \beta_k (x_i^k - v_i^k).
\end{equation}
\STATE 
Option a), $L$ is known, 
\begin{equation}
\label{eq:xkp1_opt_a_wqc}
    x_i^{k+1} = \arg\min_{x \in E} \left\{ f(y_i^k) + \langle \nabla f(y_i^k), x - y_i^k \rangle + \frac{L}{2} \|x - y_i^k \|^2 \right\}.
\end{equation}
Find $a_{k+1}$ from equation $\frac{a_{k+1}^2}{A_{k} + a_{k+1}} = \frac{1}{L}$.  \\
Option b), 
\begin{equation}
\label{eq:xkp1_opt_b_wqc}
h_{k+1} = \arg\min_{h \geqslant 0} f\left(y_i^k - h(\nabla f(y_i^k))^{\#}\right), \quad x_i^{k+1} = y_i^ k- h_{k+1}(\nabla f(y_i^k))^{\#}.
\end{equation}
Find $a_{k+1}$ from equation $f(y_i^k) - \frac{a_{k+1}^2}{2(A_{k} + a_{k+1})} \|\nabla f(y_i^k) \|_*^2 = f(x_i^{k+1})$.
\STATE  Set $A_{k+1} = A_{k} + a_{k+1}$.  
\STATE  Set $\psi_{k+1}(x) =  \psi_{k}(x) + a_{k+1}\{f(y_i^k) + \langle \nabla f(y_i^k), x - y_i^k \rangle\}$.
\STATE $v_i^{k+1} = \arg\min_{x \in E} \psi_{k+1}(x)$
\IF{$f(x_i^k) - f(x_*) \leqslant \left(1-\gamma/2\right)\left(f(x_i^0) - f(x_*)\right)$}
\STATE \textbf{break}
\ENDIF
\STATE $k = k + 1$
\ENDFOR
\STATE Set $x_{i+1}^0=x_i^N$.
\STATE $i = i + 1$
\ENDFOR
\end{algorithmic}
\end{algorithm}

Denote by $\tilde{x}^i$ the sequence of all iterates $x^j_i$ generated by the above method
\begin{theorem}
\label{Q}
If $f(x)$ is $\gamma$-weakly-quasi-convex and $L$-smooth function, then
$$
f(\tilde{x}^N) - f(x_*) = O\left(\frac{LR^2}{\gamma^3N^2}\right),
$$
where $R=\max\limits_{x:\ f(x)\leqslant f(x_0)} \|x\|.$
\end{theorem}

\begin{proof}
 Denote $\eps_0=f(x^0_0-f_*)$. From \textbf{Lemma \ref{WQC_conv}} and \textbf{Theorem \ref{Main}} we have that \[f(x_0^k)-f(x_*)\leqslant (1-\gamma)\eps_0 +\frac{2LR_0^2}{k^2},\] where $R_0=\|x_0-x_*\|_2$ We need to ensure \[f(x_0^k)-f(x_*)\leqslant (1-\gamma/2)\eps_0.\] That means that the method is first restarted no later than after $N_0=\ceil*{2\sqrt{\frac{LR^2}{\gamma\eps_0}}}$  iterations. Denote $R_1=\|x_1^0-x_*\|_2$. Again we apply \textbf{Lemma \ref{WQC_conv}} and \textbf{Theorem \ref{Main}}, we have \[f(x_1^k)-f(x_*)\leqslant (1-\gamma)(1-\gamma/2)\eps_0 +\frac{2LR_i^2}{k^2}\leqslant (1-\gamma/2)^2\eps_0,\] which implies that the second restart happens no later than after $N_1=\ceil*{2\sqrt{\frac{LR_1^2}{\gamma(1-\gamma/2)\eps_0}}}$. By proceeding in the same way we show that no more than $N_i=\ceil*{2\sqrt{\frac{LR_i^2}{\gamma(1-\gamma/2)^i\eps_0}}}$ iterations happen between the $i$-th and the $i+1$-th restarts.

Let $d=\log_{1-\gamma/2}\frac{\eps}{\eps_0}$. Then an $\eps$-solution is obtained in no more than $N=\sum\limits_{i=0}^d N_i$ iterations. We also know that the sequence $f(x^0_i)$ is non-increasing, so\\ $\forall i\ R_i\leqslant 2R=2\max\limits_{x:\ f(x)\leqslant f(x_0)} \|x\|$. It follows from our restart rule that $\eps<\eps_0(1-\gamma/2)^{d-1}$. Then we have the following sequence of relations:

\begin{align*}
N&\leqslant\sum_{i=0}^d \ceil*{2\sqrt{\frac{LR_i^2}{\gamma(1-\gamma/2)^i\eps_0}}}\leqslant d+1+\sum_{i=0}^d 2\sqrt{\frac{4LR^2}{\gamma\eps}}\left(1-\gamma/2\right)^{\frac{d-i+1}{2}}\leqslant\\
&\leqslant d+1+ 2\sqrt{\frac{4LR^2}{\gamma\eps}}\sum_{i=-\infty}^d \left(1-\gamma/2\right)^{\frac{d-i+1}{2}}=d+1+2\sqrt{\frac{4LR^2}{\gamma\eps}}\frac{\sqrt{1-\gamma/2}}{1-\sqrt{1-\gamma/2}}=\\
&=d+1+2\sqrt{4\frac{LR^2}{\gamma\eps}}\frac{\sqrt{1-\gamma/2}(1+\sqrt{1-\gamma/2})}{\gamma/2}=d+1+3\sqrt{\frac{4LR^2}{\gamma^3\eps}}=O\left(\sqrt{\frac{LR^2}{\gamma^3\eps}}\right).
\end{align*}
\end{proof}

Note, that using the Sequential Subspace Optimization Method \cite{SESOP} Guminov et al.\cite{guminov2017accelerated} show that the last bound in theorem \ref{Q} can be improved under small $\gamma$ to
$$f(\tilde{x}^N) - f(x_*) = O\left(\frac{LR^2}{\gamma^2 N^2}\right).$$ However, this requires solving a three-dimensional non-convex problem on each iteration. The method in this paper, on the other hand, only requires solving one minimization problem over an interval. If $R$ is known, the stopping criterion may be used to restart the method.
\subsection{Strongly convex objectives}
Assume now that the objective function in problem \eqref{COnvProb} is $\mu$-strongly convex with respect to the Euclidean norm:
\[\forall x,y\quad f(y)\geqslant f(x)+\langle\nabla f(x),y-x\rangle+\|y-x\|^2_2.\]Next we describe two different ways to modify or method in order to deal with strongly convex objective functions.

The first way is to consider a slightly different estimating sequence:
$$\psi_{k+1}(x) =  \psi_{k}(x) + a_{k+1}\{f(y^k) + \langle \nabla f(y^k), x - y^k \rangle+\frac{\mu}{2}\|x-y^k\|^2\}.$$ This leads us to the following method.

\begin{algorithm}[!h]
\caption{Accelerated Gradient Method with Small-Dimensional Relaxation (AGMsDR)}
\label{AGMsDR-SC}
\begin{algorithmic}[1]
\ENSURE $x^k$
\STATE Set $k = 0$, $A_0=0$, $x^0 = v^0$, $\psi_0(x) = \frac{1}{2}\|x-x_0\|^2_2$, $\tau_0=1$
\FOR{$k \geqslant 0$}
\STATE \begin{equation}
    \label{eq:beta_k_y_k_def_sc}
    \beta_k = \arg\min_{\beta \in \left[0, 1 \right]} f\left(v^k + \beta (x^k - v^k)\right), \quad y^k = v^k + \beta_k (x^k - v^k).
\end{equation}
\STATE 
Option a), $L$ is known, 
\begin{equation}
\label{eq:xkp1_opt_a_sc}
    x^{k+1} = \arg\min_{x \in E} \left\{ f(y^k) + \langle \nabla f(y^k), x - y^k \rangle + \frac{L}{2} \|x - y^k \|^2 \right\}.
\end{equation}
Find $a_{k+1}$ from equation $\frac{a_{k+1}^2}{(\tau_{k}+\mu a_{k+1})(A_{k} + a_{k+1})} = \frac{1}{L}$.  \\
Option b), 
\begin{equation}
\label{eq:xkp1_opt_b_sc}
h_{k+1} = \arg\min_{h \geqslant 0} f\left(y^k - h(\nabla f(y^k))^{\#}\right), \quad x^{k+1} = y^  k- h_{k+1}(\nabla f(y^k))^{\#}.
\end{equation}
Find $a_{k+1}$ from equation $f(y^k) - \frac{a_{k+1}^2}{2(\tau_{k}+\mu a_{k+1})(A_k+a_{k+1})} \|\nabla f(y^k) \|_2^2+\frac{\mu\tau_k a_{k+1}}{2(\tau_{k}+\mu a_{k+1})(A_k+a_{k+1})}\|v^k-y^k\|^2=f(x^{k+1})$.
\STATE  Set $A_{k+1} = A_{k} + a_{k+1}$, $\tau_{k+1}=\tau_k+\mu a_{k+1}$.  
\STATE  Set $\psi_{k+1}(x) =  \psi_{k}(x) + a_{k+1}\{f(y^k) + \langle \nabla f(y^k), x - y^k \rangle+\frac{\mu}{2}\|x-y^k\|^2\}.$
\STATE $v^{k+1} = \arg\min_{x \in E} \psi_{k+1}(x)$
\STATE $k = k + 1$
\ENDFOR
\end{algorithmic}
\end{algorithm}

\begin{theorem} 
\label{Main_SC}
After $k$ steps of Algorithm \ref{AGMsDR-SC} for problem \eqref{COnvProb} it holds that
\begin{equation}
\label{eq:main_recurrence_sc}
        A_{k}f(x^{k}) \leqslant \min_{x \in \mathbb{R}^n} \psi_{k}(x) = \psi_{k}(v^{k}).
\end{equation}
Moreover, \[A_k\geqslant\max\left\lbrace \frac{k^2}{4L},\frac{1}{L}\left(1-\sqrt{\frac{\mu}{L}}\right)^{-(k-1)}\right\rbrace.\]
\end{theorem}
\begin{proof} 
Denote
$$l_k(x) = \sum_{i=0}^k a_{i+1}\{f(y^i) + \langle \nabla f(y^i), x - y^i \rangle +\frac{\mu}{2}\|x-y^k\|^2\}.$$
Then 
$$\psi_{k+1}(x) = l_{k}(x) + \psi_0(x) = \psi_{k}(x) + a_{k+1}\{f(y^k) + \langle \nabla f(y^k), x - y^k \rangle+\frac{\mu}{2}\|x-y^k\|^2\}.$$

Note that $\psi_{k}$ is a sum of a 1-strongly convex function $\psi_0$ and $\mu a_i$-strongly convex functions for $i=1,\ldots,k$, which means that $\psi_{k}$ is $\tau_{k}$-strongly convex, where $\tau_{k}=1+\mu\sum\limits_{i=1}^{k}a_i=1+\mu A_{k}$.

First, we prove inequality \eqref{eq:main_recurrence_sc} by induction over $k$. For $k=0$, the inequality holds. Assume that 
$$A_{k}f(x^{k}) \leqslant \min_{x \in \mathbb{R}^n} \psi_{k}(x) = \psi_{k}(v^{k}).$$
Then 
$$\psi_{k+1}(v^{k+1}) = \min_{x \in \mathbb{R}^n} \left\{ \psi_{k}(x) + a_{k+1}\{f(y^k) + \langle \nabla f(y^k), x - y^k \rangle+\frac{\mu}{2}\|x-y^k\|^2\} \right\}  $$
$$  \geqslant \min_{x \in \mathbb{R}^n} \left\{ \psi_{k}(v^k) +\frac{\tau_k}{2}\|x-v^k\|^2 + a_{k+1}\{f(y^k) + \langle \nabla f(y^k), x - y^k \rangle+\frac{\mu}{2}\|x-y^k\|^2\} \right\} $$
$$\geqslant \min_{x \in \mathbb{R}^n} \left\{ A_{k}f(x^k) +\frac{\tau_k}{2}\|x-v^k\|^2 + a_{k+1}\{f(y^k) + \langle \nabla f(y^k), x - y^k \rangle+\frac{\mu}{2}\|x-y^k\|^2\} \right\}.$$
Here we used that $\psi_{k}$ is a $\tau_k$-strongly convex function with minimum at $v^k$.

By the definition of $\beta_k$ and $y_k$ in \eqref{eq:beta_k_y_k_def_sc}, we have 
$f(y^k) \leqslant f(x^k)$. By the optimality conditions in \eqref{eq:beta_k_y_k_def_sc}, either
\begin{enumerate}
    \item $\beta_k = 0$, $\langle \nabla f(y^k),x^k - v^k \rangle \geqslant 0$, $y^k = v^k$;
    \item $\beta_k \in (0,1)$ and $\langle \nabla f(y^k),x^k - v^k \rangle = 0$, $y^k = v^k + \beta_k (x^k - v^k)$;
    \item $\beta_k = 1$ and $\langle \nabla f(y^k),x^k - v^k \rangle \leqslant 0$, $y^k = x^k$ .
\end{enumerate}
In all three cases,  $\langle \nabla f(y^k), v^k - y^k \rangle \geqslant 0$. Thus,
$$\psi_{k+1}(v^{k+1}) \geqslant \min_{x \in \mathbb{R}^n} \left\{ A_{k}f(y^k) +\frac{\tau_k}{2}\|x-v^k\|^2 + a_{k+1}\{f(y^k) + \langle \nabla f(y^k), x - y^k \rangle+\frac{\mu}{2}\|x-y^k\|^2\} \right\}. $$

The explicit solution to this quadratic minimization problem is \[x=\frac{1}{\tau_{k+1}}(\tau_k v^k+\mu a_{k+1}y^k-a_{k+1}\nabla f(y^k)).\] By plugging in the solution and using $\langle \nabla f(y^k), v^k - y^k \rangle \geqslant 0$, we obtain 

\[\psi_{k+1}(v^{k+1})\geqslant  A_{k+1}f(y^k) - \frac{a_{k+1}^2}{2\tau_{k+1}}\|\nabla f(y^k)\|_2^2+\frac{\mu\tau_k a_{k+1}}{2\tau_{k+1}}\|v^k-y^k\|^2.\]
Our next goal is to show that
\begin{equation} 
\label{eq:ind_fin_step_sc}
A_{k+1}f(y^k) - \frac{a_{k+1}^2}{2\tau_{k+1}}\|\nabla f(y^k)\|_2^2+\frac{\mu\tau_k a_{k+1}}{2\tau_{k+1}}\|v^k-y^k\|^2 \geqslant A_{k+1}f(x^{k+1}),
\end{equation}
which proves the induction step.

For option a) in the step 4, \eqref{grad_descent_guarantee} takes the form
\begin{equation}
 f(x^{k+1}) \leqslant f(y^k) - \frac{1}{2L} \|\nabla f(y^k)\|^2.
 \label{grad_descent_guarantee_2}
\end{equation}

Since, for this option, $\frac{a_{k+1}^2}{(\tau_{k}+\mu a_{k+1})(A_{k} + a_{k+1})} = \frac{1}{L}$, inequality \eqref{eq:ind_fin_step_sc} holds.
For option b) in the step 4, \eqref{eq:ind_fin_step_sc} holds by the choice of $a_{k+1}$ from the equation
\begin{equation}
    \label{eq:Th:Main_sc_proof_1}
    f(y^k) - \frac{a_{k+1}^2}{2A_{k+1}\tau_{k+1}} \|\nabla f(y^k) \|_2^2+\frac{\mu\tau_k a_{k+1}}{2A_{k+1}\tau_{k+1}}\|v^k-y^k\|^2=f(x^{k+1}).
\end{equation}
It remains to show that this equation has a solution $a_{k+1} > 0$.
Since $A_{k+1}=A_k + a_{k+1}$ and $\tau_{k+1}=\tau_k+\mu a_{k+1}$, we can rewrite the equation \eqref{eq:Th:Main_sc_proof_1} as 
\[
(2\mu\delta_k+\|\nabla f(y^k)\|^2_2)a_{k+1}^2+(2\delta_k(\tau_k+\mu A_k)-\mu\tau_k\|v^k-y^k\|^2_2)a_{k+1}+2\tau_k A_k\delta_k=0,
\]
where $\delta_k=f(x^{k+1})-f(y^k)<0$.
By strong convexity, $f(y^{k}) - f(x_*) \leqslant \frac{1}{2\mu}\|\nabla f(y^k)\|^2_2$, we have \[2\mu\delta_k+\|\nabla f(y^k)\|^2_2\geqslant 2\mu(f(x^{k+1})-f(x_*))\geqslant 0.\] Therefore, a non-negative solution exists and may be written down as
\[a_{k+1}=\frac{-S_k+\sqrt{S_k^2-8A_k\delta_k\tau_k(2\delta_{k}\mu+\|\nabla f(y_k)\|^2)}}{4\delta_{k}\mu+2\|\nabla f(y_k)\|^2},\]
where $S_k=2\delta_{k}(\tau_k+\mu A_k)-\mu\tau_k\|v^k-y^k\|^2$

Let us estimate the rate of the growth for $A_k$. If in the step 4 option a) is used, $\frac{a_{k+1}^2}{\tau_{k+1}A_{k+1}} = \frac{1}{L}$. For the option b), using \eqref{eq:Th:Main_sc_proof_1} and \eqref{grad_descent_guarantee_2}, we have 
\[f(y^k) - \frac{a_{k+1}^2}{2A_{k+1}\tau_{k+1}} \|\nabla f(y^k) \|_2^2+\frac{\mu\tau_k a_{k+1}}{2A_{k+1}\tau_{k+1}}\|v^k-y^k\|^2\leqslant f(y^k)-\frac{1}{2L}\|\nabla f(y^k)\|^2_2.\] Thus, for both options, $\frac{a_{k+1}^2}{\tau_{k+1}A_{k+1}} \geqslant \frac{1}{L}$, or \[a_i\geqslant\frac{1}{\sqrt{L}}\sqrt{A_i+\mu A_i^2}\geqslant\sqrt{\frac{\mu}{L}}A_i.\] Using the left inequality, we obtain

\begin{equation}\label{sqrt_A}
    \sqrt{A_{i}}-\sqrt{A_{i-1}}\geqslant \frac{A_{i}-A_{i-1}}{\sqrt{A_{i}}+\sqrt{A_{i-1}}}\geqslant \frac{a_{i}}{2\sqrt{A_{i}}}\geqslant \frac{1}{2\sqrt{L}}\sqrt{1+\mu A_{i}}.
\end{equation} This in turn implies a weaker inequality

\[\sqrt{A_{i}}-\sqrt{A_{i-1}}\geqslant \frac{1}{2\sqrt{L}}.\] Summing it up for $i=1,\ldots, k$ we get \[A_{k}\geqslant \frac{k^2}{4L}.\]

We also have \[A_{k+1}=A_k+a_{k+1}\geqslant A_k+\sqrt{\frac{\mu}{L}}A_{k+1},\] which leads to \[A_{k+1}\geqslant\left(1-\sqrt{\frac{\mu}{L}}\right)^{-1}A_k.\] To use this bound we only need to estimate $A_1$, which we can do as follows: \[A_1=\frac{a_1^2}{A_1}\geqslant\frac{a_{1}^2}{(1+\mu A_1)A_{1}}=\frac{a_{1}^2}{\tau_{1}A_{1}} \geqslant \frac{1}{L}\]

By recursively applying the last bound we reach the desired result:

\[A_k\geqslant\max\left\lbrace \frac{k^2}{4L},\frac{1}{L}\left(1-\sqrt{\frac{\mu}{L}}\right)^{-(k-1)}\right\rbrace\]

\end{proof}

\begin{theorem}
\label{SC_conv}
Let function $f$ be $\mu$-strongly convex and $L$-smooth and Algorithm \ref{AGMsDR} be run for $N$ steps. Then \[f(x_{k})-f(x_*)\leqslant\min\left\lbrace\frac{2LR^2}{k^2},\left(1-\sqrt{\frac{\mu}{L}}\right)^{-(k-1)}LR^2\right\rbrace,\]

where $R=\|x_0-x_*\|$
\end{theorem}
\begin{proof}
According to the definition of $\mu$-strong convexity 
$$l_k(x_*) = \sum_{i=0}^k a_{i+1}\{f(y^i) + \langle \nabla f(y^i), x_* - y^i \rangle+\frac{\mu}{2}\|x_*-y^i\|^2_2\leqslant\sum_{i=0}^k a_{i+1}f(x_*)=A_{k+1}f(x_*).$$
From this inequality and \textbf{Theorem} \ref{Main_SC} we have
$$A_{k} f(x_{k}) \leqslant \min_{x \in \mathbb{R}^n} \psi_{k}(x) \leqslant \psi_{k}(x_*) = l_{k-1}(x_*) + \frac{1}{2}\|x_0-x_*\|^2_2 \leqslant A_k f(x_*) + \frac{1}{2}\|x_0-x_*\|^2_2.$$

Finally, denoting $R=\|x_0-x_*\|$, we have \[f(x_{k})-f(x_*)\leqslant\min\left\lbrace\frac{2LR^2}{k^2},\left(1-\sqrt{\frac{\mu}{L}}\right)^{(k-1)}LR^2\right\rbrace.\]
\end{proof}

Another way to apply the algorithm to strongly convex objective is to use a restart procedure.
\begin{algorithm}[!h]
\caption{Accelerated Gradient Method with Small-Dimensional Relaxation (AGMsDR)}
\label{AGMsDR_SC_Restarted}
\begin{algorithmic}[1]
\ENSURE $x^k_i$
\FOR{$i \geqslant 0$}
\STATE Set $k = 0$, $A_0=0$, $x_i^0 = v_i^0$, $\psi^i_0(x) = \frac{1}{2}\|x-x_0\|^2_2$
\FOR{$k \geqslant 0$}
\STATE \begin{equation}
    \label{eq:beta_k_y_k_def_sc_Restarted}
    \beta_k = \arg\min_{\beta \in \left[0, 1 \right]} f\left(v_i^k + \beta (x_i^k - v_i^k)\right), \quad y_i^k = v_i^k + \beta_k (x_i^k - v_i^k).
\end{equation}
\STATE 
Option a), $L$ is known, 
\begin{equation}
\label{eq:xkp1_opt_a_sc_Restarted}
    x_i^{k+1} = \arg\min_{x \in E} \left\{ f(y_i^k) + \langle \nabla f(y_i^k), x - y_i^k \rangle + \frac{L}{2} \|x - y_i^k \|^2 \right\}.
\end{equation}
Find $a_{k+1}$ from equation $\frac{a_{k+1}^2}{A_{k} + a_{k+1}} = \frac{1}{L}$.  \\
Option b), 
\begin{equation}
\label{eq:xkp1_opt_b_sc_Restarted}
h_{k+1} = \arg\min_{h \geqslant 0} f\left(y_i^k - h(\nabla f(y_i^k))^{\#}\right), \quad x_i^{k+1} = y_i^ k- h_{k+1}(\nabla f(y_i^k))^{\#}.
\end{equation}
Find $a_{k+1}$ from equation $f(y_i^k) - \frac{a_{k+1}^2}{2(A_{k} + a_{k+1})} \|\nabla f(y_i^k) \|_*^2 = f(x_i^{k+1})$.
\STATE  Set $A_{k+1} = A_{k} + a_{k+1}$.  
\STATE  Set $\psi_{k+1}(x) =  \psi_{k}(x) + a_{k+1}\{f(y_i^k) + \langle \nabla f(y_i^k), x - y_i^k \rangle\}$.
\STATE $v_i^{k+1} = \arg\min_{x \in E} \psi_{k+1}(x)$
\IF{$\|x_i^k-x_*\|_2^2\leqslant\frac{1}{2}\|x_i^0-x_*\|_2^2$ 
\iffalse\pd{We can't check this inequality as we do not know $x_*$. The idea is to use inequality $\frac{\mu}{2}\|x^k-x_*\|^2 \leqslant (x^k)-f(x_*)\leqslant \frac{R^2}{2A_k}$ . Whenever the r.h.s. is smaller than $\frac{\mu}{4}$, we have that $\|x_i^k-x_*\|_2^2\leqslant\frac{1}{2}\|x_i^0-x_*\|_2^2$.}\fi}
\STATE \textbf{break}
\ENDIF
\STATE $k = k + 1$
\ENDFOR
\STATE Set $x_{i+1}^0=x_i^N$.
\STATE $i = i + 1$
\ENDFOR
\end{algorithmic}
\end{algorithm}

Of course, we have no direct way to check the inequality in step 9 of the algorithm. However, using strong convexity, we have that $\frac{\mu}{2}\|x^k-x_*\|^2 \leqslant f(x^k)-f(x_*)\leqslant \frac{R^2}{2A_k}$. Provided $\mu$ is known, it is sufficient to check whether the r.h.s. is smaller than $\frac{\mu}{4}R^2$, which would imply $\|x_i^k-x_*\|_2^2\leqslant\frac{1}{2}\|x_i^0-x_*\|_2^2$.
\begin{theorem}
\label{Main_sc_restated}
If $f(x)$ is a $\mu$-strongly convex and $L$-smooth function, then
\[\|\tilde{x}^N-x_*\|_2^2=O(2^{-\sqrt{\mu N^2/L}}R^2),\] \[
f(\tilde{x}^N) - f(x_*) = O(2^{-\sqrt{\mu N^2/L}}LR^2).
\]
where $R=\|x_0-x^*\|.$

\end{theorem}

\begin{proof}
 From the very definition of strong convexity we have \[\frac{\mu}{2}\|x_i^k-x_*\|^2_2\leqslant f(x_i^k)-f(x_*).\] From  \textbf{Theorem \ref{Th:RateConv}} we have that \[f(x_i^k)-f(x_*)\leqslant \frac{2L\|x_i^0-x_*\|_2^2}{k^2}.\] To ensure $\|x_i^k-x_*\|_2^2\leqslant\frac{1}{2}\|x_i^0-x_*\|_2^2$ we then need to satisfy the following inequality:
 \[\frac{4L\|x_i^0-x_*\|_2^2}{\mu k^2}\leqslant\frac{1}{2}\|x_i^0-x_*\|_2^2.\]
 That means that no more than $N_i=\ceil*{\sqrt{8L/\mu}}$ iterations happen between the $i$-th and the $i+1$-th restarts.

Hence, \[\|\tilde{x}^N-x_*\|_2^2=O(2^{-\sqrt{\mu N^2/L}}R^2),\] \[
f(\tilde{x}^N) - f(x_*) = O(2^{-\sqrt{\mu N^2/L}}LR^2).
\]

\end{proof}
\subsection{Implementation details: line-search accuracy}

In our methods the line search step is used to perform linear coupling and steepest descent. It will now be shown that in both cases performing the line search exactly is not critical for the methods' convergence.

\begin{itemize}
\item \textbf{Steepest descent}. In algorithms APDLSGD, UAPDLSGD and SCUAPDLSGD steepest descent is used to construct $x^{k+1}$. However, the convergence analysis of all the above methods only relies on $f(x^{k+1})$ being no greater than $f(y^k-\frac{1}{l}\nabla f(y^k))$, where $l=L$ for the APDLSGD method and $l=M(\frac{a_{k+1}}{A_{k+1}}\eps, \nu, M_\nu)$ for the universal methods. This means that the accuracy of line search has no effect on the worst-case convergence bounds, as long as it is good enough to ensure that the result is no worse than one obtained by performing a gradient descent step. Since for most objectives using exact steepest descent should result in iterates different from ones obtained by gradient descent, it is reasonable to expect that performing steepest descent with some small error will still lead to iterates with low enough objective values. 

\item\textbf{Linear coupling}. The fact that the step \[\beta_k = \arg\min_{\beta \in \left[0, 1 \right]} f\left(v^k + \beta (x^k - v^k)\right);\ y^k = v^k + \beta_k (x^k - v^k) \] guarantees $f(y^k)\leqslant f(x^k)$ and $\la\nabla f(y^k),v^k-y^k\ra\geqslant 0$ is used in the convergence analysis. Again, it is reasonable to expect the first inequality to hold true in the case of inexact line search. We will now show that allowing for some error in the second inequality does not lead to accumulating errors.  
\end{itemize}

\begin{lemma}
For the not necessarily convex problem \eqref{COnvProb} and the APDGD method with step 3 performed in a way that guarantees $f(y^k)\leqslant f(x^k)$ and $\la\nabla f(y^k),v^k-y^k\ra\geqslant -\tilde{\eps},$
$$A_{k+1}f(x^{k+1}) \leqslant \min_{x \in \mathbb{R}^n} \psi_{k+1}(x) + A_{k+1}\tilde{\eps} = \psi_{k+1}(v^{k+1}) + A_{k+1}\tilde{\eps}$$
and
$$A_k = O\left(\frac{k^2}{L}\right).$$
\end{lemma}

\begin{proof}
Theorem can be prove by induction. Let's consider the step of induction. That is, assume that we've already proved that 
$$A_{k}f(x^{k}) \leqslant \min_{x \in \mathbb{R}^n} \psi_{k}(x) +A_{k}\tilde{\eps} = \psi_{k}(v^{k})+A_{k}\tilde{\eps}.$$
Then 
$$\psi_{k+1}(v^{k+1}) = \min_{x \in \mathbb{R}^n} \left\{ \psi_{k}(x) + a_{k+1}\{f(y^k) + \langle \nabla f(y^k), x - y^k \rangle\} \right\} =$$
$$= \min_{x \in \mathbb{R}^n} \left\{ \psi_{k}(v^k) +\frac{1}{2}\|x-v^k\|_2^2 + a_{k+1}\{f(y^k) + \langle \nabla f(y^k), x - y^k \rangle\} \right\} \geqslant$$
$$\geqslant \min_{x \in \mathbb{R}^n} \left\{ A_{k}f(x^k) +\frac{1}{2}\|x-v^k\|_2^2 + a_{k+1}\{f(y^k) + \langle \nabla f(y^k), x - y^k \rangle\} -A_k\tilde{\eps} \right\}.$$
Due to the line 3 $f(y^k) \leqslant f(x^k)$ and $\langle \nabla f(y^k), v^k - y^k \rangle \geqslant -\tilde{\eps}$. From this two inequalities and the fact that $A_{k+1} = A_{k} +a_{k+1}$ one can obtain
$$\psi_{k+1}(v^{k+1}) \geqslant \min_{x \in \mathbb{R}^n} \left\{ A_{k+1}f(y^k) + a_{k+1}\langle \nabla f(y^k), x - v^k \rangle +\frac{1}{2}\|x-v^k\|_2^2-A_{k+1}\tilde{\eps}  \right\} = $$
$$= A_{k+1}f(y^k) - \frac{a_{k+1}^2}{2}\|\nabla f(y^k)\|_2^2-A_{k+1}\tilde{\eps}$$
So let's choose $a_{k+1}$ in such a way that guarantee 
\begin{equation}
A_{k+1}f(y^k) - \frac{a_{k+1}^2}{2}\|\nabla f(y^k)\|_2^2 = A_{k+1}f(x^{k+1}).
\end{equation}
This is quadratic equation on $a_{k+1}$. One can solve it explicitly.
For the method APDGD (see line 4) this equation means that $\frac{a_{k+1}^2}{A_{k+1}} \geqslant \frac{1}{L}$, which, combined with $A_{k+1} = A_{k} + a_{k+1}$, means that $a_k = O\left(\frac{k}{L}\right)$, $A_k = O\left(\frac{k^2}{L}\right)$.

\end{proof}

Of course, the same result applies to all the other version of the method presented further. This lemma leads to an additive term $\tilde{\eps}$ in the convergence bounds.

\section{Universal methods}
\label{S:univ}
We consider the optimization problem \eqref{COnvProb}. By considering the class of objectives with H\"{o}lder continuous (sub)gradients we may generalize the above methods to non-smooth problems.

In this section we assume that the objective function has H\"{o}lder continuous (sub)gradients: for all $x, y \in \mathbb{R}^n$ and some $\nu\in\left[0,1\right]$
\begin{equation}\label{eq:hold_cond}
\|\nabla f(y) - \nabla f(x)\|_* \leqslant M_\nu \|x-y\|^\nu.
\end{equation}

Here if $\nu=0$ $\nabla f(x)$ denotes some subgradient of $f(x)$.

Again, we will be using the sequence of estimating functions defined as
$$\psi_0(x) = V(x,x^0).$$
$$l_k(x) = \sum_{i=0}^k a_{i+1}\{f(y^i) + \langle \nabla f(y^i), x - y^i \rangle\},$$
$$\psi_{k+1}(x) = l_{k}(x) + \psi_0(x) = \psi_{k}(x) + a_{k+1}\{f(y^k) + \langle \nabla f(y^k), x - y^k \rangle\},$$

$$A_{k+1} = A_{k} + a_{k+1}, \quad A_0 = 0.$$

\begin{algorithm}[!h]
\caption{Universal Accelerated Gradient Method with Small-Dimensional Relaxation (UAGMsDR)}
\label{UAGMsDR}
\begin{algorithmic}[1]
\REQUIRE Accuracy $\eps$
\ENSURE $x^k$
\STATE Set $k = 0$, $A_0=0$, $x^0 = v^0$, $\psi_0(x) = V(x,x^0)$
\FOR{$k \geqslant 0$}
\STATE \begin{equation}
    \label{eq:beta_k_y_k_def_ns}
    \beta_k = \arg\min_{\beta \in \left[0, 1 \right]} f\left(v^k + \beta (x^k - v^k)\right), \quad y^k = v^k + \beta_k (x^k - v^k).
\end{equation}
\STATE 

\begin{equation}
\label{eq:xkp1_ns}
h_{k+1} = \arg\min_{h \geqslant 0} f\left(y^k - h(\nabla f(y^k))^{\#}\right), \quad x^{k+1} = y^  k- h_{k+1}(\nabla f(y^k))^{\#}.
\end{equation}
Find $a_{k+1}$ from equation $f(y^k) - \frac{a_{k+1}^2}{2(A_{k} + a_{k+1})} \|\nabla f(y^k) \|_*^2  + \frac{\varepsilon a_{k+1}}{2(A_{k} + a_{k+1})} = f(x^{k+1})$.
\STATE  Set $A_{k+1} = A_{k} + a_{k+1}$.  
\STATE  Set $\psi_{k+1}(x) =  \psi_{k}(x) + a_{k+1}\{f(y^k) + \langle \nabla f(y^k), x - y^k \rangle\}$.
\STATE $v^{k+1} = \arg\min_{x \in E} \psi_{k+1}(x)$
\STATE $k = k + 1$
\ENDFOR
\end{algorithmic}
\end{algorithm}

In the analysis of the above method we will be using a particular choice of the subgradient in step 4. However, it seems that in practice this is not important for the method's convergence.
All the results mentioned below remain correct if the line search domain $[0,1]$ in line 3 is changed to any larger subset of $\mathbb{R}$. 
Note that unlike other universal methods, (\cite{nesterov2015universal,guminov2017universal}) this method does not require estimating the step length in an inner cycle. This results in a slightly better complexity bound due to better step-lengths and lower iteration complexity.

The following lemma (the proof of which may be found in \cite{nesterov2015universal}) plays a major role in the convergence analysis of this method.
\begin{lemma}\label{inexact}
Let function $f(x)$ have H\"{o}lder continuous (sub)gradients for some $\nu\in[0,1]$ and $M_\nu<+\infty$. Then for any $\delta>0$ we have \[f(y)\leqslant f(x)+\langle\nabla f(x), y-x\rangle+\frac{M}{2}\|y-x\|^2+\frac{\delta}{2},\] where \[M=M\left(\delta,\nu, M_\nu\right)=\left[\frac{1-\nu}{1+\nu}\frac{M_\nu}{\delta}\right]^{\frac{1-\nu}{1+\nu}}M_\nu.\]

If the subgradient is not H\"{o}lder continuous for some exponent $\nu$ it is convenient to consider the corresponding $M_\nu$ to be equal to $+\infty$.
\end{lemma}

\begin{theorem}\label{universal-Ak-rate}
For the algorithm UAGMsDR and possibly non-convex  \eqref{COnvProb}, where $f(x)$ has H\"{o}lder continuous (sub)gradients,
\begin{equation}\label{eq:main_recurrence_universal}
    A_k f(x^k) \leqslant \min_{x \in \mathbb{R}^n} \psi_{k}(x) +\frac{A_{k}\eps}{2} = \psi_{k}(v^{k})+\frac{A_{k}\eps}{2}.
\end{equation}
and
\[A_k\geqslant\sup_{\nu\in[0,1]} \left[\frac{1+\nu}{1-\nu}\right]^\frac{1-\nu}{1+\nu}\frac{k^\frac{1+3\nu}{1+\nu}\eps^\frac{1-\nu}{1+\nu}}{2^\frac{1+3\nu}{1+\nu}M_\nu^\frac{2}{1+\nu}} \]
\end{theorem}

\begin{proof}
Denote
$$l_k(x) = \sum_{i=0}^k a_{i+1}\{f(y^i) + \langle \nabla f(y^i), x - y^i \rangle\}.$$
Then 
$$\psi_{k+1}(x) = l_{k}(x) + \psi_0(x) = \psi_{k}(x) + a_{k+1}\{f(y^k) + \langle \nabla f(y^k), x - y^k \rangle\}.$$
First, we prove inequality \eqref{eq:main_recurrence_universal} by induction over $k$. For $k=0$, the inequality holds. Assume that 
$$A_k f(x^k) \leqslant \min_{x \in \mathbb{R}^n} \psi_{k}(x) +\frac{A_{k}\eps}{2} = \psi_{k}(v^{k})+\frac{A_{k}\eps}{2}.$$
Then 
$$\psi_{k+1}(v^{k+1}) = \min_{x \in \mathbb{R}^n} \left\{ \psi_{k}(x) + a_{k+1}\{f(y^k) + \langle \nabla f(y^k), x - y^k \rangle\} \right\}  $$
$$  \geqslant \min_{x \in \mathbb{R}^n} \left\{ \psi_{k}(v^k) +\frac{1}{2}\|x-v^k\|^2 + a_{k+1}\{f(y^k) + \langle \nabla f(y^k), x - y^k \rangle\} \right\} $$
$$\geqslant \min_{x \in \mathbb{R}^n} \left\{ A_{k}f(x^k)-\frac{A_{k}\eps}{2} +\frac{1}{2}\|x-v^k\|^2 + a_{k+1}\{f(y^k) + \langle \nabla f(y^k), x - y^k \rangle\} \right\}.$$
Here we used that $\psi_{k}$ is a strongly convex function with minimum at $v^k$.

By the definition of $\beta_k$ and $y_k$ in \eqref{eq:beta_k_y_k_def_ns}, we have 
$f(y^k) \leqslant f(x^k)$. By the optimality conditions in \eqref{eq:beta_k_y_k_def_ns}, there exists such a subgradient $\nabla f(y^k)$ that either
\begin{enumerate}
    \item $\beta_k = 0$, $\langle \nabla f(y^k),x^k - v^k \rangle \geqslant 0$, $y^k = v^k$;
    \item $\beta_k \in (0,1)$ and $\langle \nabla f(y^k),x^k - v^k \rangle = 0$, $y^k = v^k + \beta_k (x^k - v^k)$;
    \item $\beta_k = 1$ and $\langle \nabla f(y^k),x^k - v^k \rangle \leqslant 0$, $y^k = x^k$ .
\end{enumerate}
In all three cases,  $\langle \nabla f(y^k), v^k - y^k \rangle \geqslant 0$. Thus,
$$\psi_{k+1}(v^{k+1}) \geqslant \min_{x \in \mathbb{R}^n} \left\{ A_{k+1}f(y^k)-\frac{A_{k}\eps}{2} + a_{k+1}\langle \nabla f(y^k), x - v^k \rangle +\frac{1}{2}\|x-v^k\|^2  \right\}  $$
$$\geqslant A_{k+1}f(y^k)-\frac{A_{k}\eps}{2} - \frac{a_{k+1}^2}{2}\|\nabla f(y^k)\|_*^2,$$
where we used that for any $g \in E^*$, $s \in E$, $\zeta \geqslant 0$, it holds that $\la g,s \ra + \frac{\zeta}{2}\|s\|^2 \geqslant - \frac{1}{2\zeta}\|g\|_*^2$. 

The equation 
\begin{equation} 
\label{eq:ind_fin_step_universal}
A_{k+1}f(y^k) - \frac{a_{k+1}^2}{2}\|\nabla f(y^k)\|_*^2 \geqslant A_{k+1}f(x^{k+1})-\frac{a_{k+1}\eps}{2},
\end{equation}
holds by the choice of $a_{k+1}$ from the equation
\begin{equation}
    \label{eq:Th:Main_proof_1_universal}
    f(y^k) - \frac{a_{k+1}^2}{2A_{k+1}} \|\nabla f(y^k) \|_*^2+\frac{a_{k+1}\eps}{2A_{k+1}} = f(x^{k+1}).
\end{equation}
It remains to show that this equation has a solution $a_{k+1} > 0$.
Applying \textbf{Lemma}~\ref{inexact} with $\delta=\frac{a_{k+1}\eps}{A_{k+1}}$, we have 
\begin{align}
    \label{eq:Th:Main_proof_2_universal}
f(x^{k+1}) &= \min_{h \geqslant 0} f\left(y^k - h(\nabla f(y^k))^{\#}\right) \leqslant \\ 
&\leqslant \min_{h \geqslant 0} \left( f(y^k) -h \langle \nabla f(y^k), (\nabla f(y^k))^{\#} \rangle + \frac{Mh^2}{2} \|(\nabla f(y^k))^{\#} \|^2 +\frac{a_{k+1}\eps}{2A_{k+1}}\right) \notag \\
 &= f(y^k) - \frac{1}{2M} \|\nabla f(y^k)\|_*^2+\frac{a_{k+1}\eps}{2A_{k+1}},
\end{align}
where $M=\left[\frac{1-\nu}{1+\nu}\frac{A_{k+1}M_\nu}{a_{k+1}\eps}\right]^{\frac{1-\nu}{1+\nu}}M_\nu$, and we used that $\langle \nabla f(y^k), (\nabla f(y^k))^{\#} \rangle = \|\nabla f(y^k) \|_*^2$ and $\|(\nabla f(y^k))^{\#} \|^2 = 1$ by definition of the vector $(\nabla f(y^k))^{\#}$. Since $A_{k+1}=A_k + a_{k+1}$, we can rewrite the equation \eqref{eq:Th:Main_proof_1_universal} as 
\[
\frac{a_{k+1}^2}{2} \|\nabla f(y^k) \|_*^2 + a_{k+1} \left(f(x^{k+1}) - f(y^k)-\frac{\eps}{2}\right) +  A_k(f(x^{k+1}) - f(y^k)) = 0.
\]
Since, by \eqref{eq:Th:Main_proof_2_universal}, $f(x^{k+1}) - f(y^k) < \frac{a_{k+1}\eps}{2A_{k+1}}$ (otherwise $\|\nabla f(y^k)\|_* = 0$ and $y_k$ is a solution to the problem \eqref{COnvProb}), at least one solution exists, and the greater one is
\[a_{k+1}=\frac{-(f(x^{k+1})-f(y^k)-\eps/2)+\sqrt{(f(x^{k+1})-f(y^k)-\eps/2)^2-2A_k(f(x^{k+1})-f(y^k))\|\nabla f(y_k)\|_*^2}}{\|\nabla f(y_k)\|_*^2}.\]

Let us estimate the rate of the growth for $A_k$. Using \eqref{eq:Th:Main_proof_1_universal} and \eqref{eq:Th:Main_proof_2_universal}, we have 
\[\frac{a^2_k}{A_k}\geqslant\frac{1}{M_\nu}\left[\frac{1+\nu}{1-\nu}\frac{\eps}{M_\nu}\right]^{\frac{1-\nu}{1+\nu}}\left[\frac{a_k}{A_k}\right]^{\frac{1-\nu}{1+\nu}},\]

or \[\frac{a^\frac{1+3\nu}{1+\nu}_k}{A^\frac{2\nu}{1+\nu}_k}\geqslant\frac{1}{M_\nu}\left[\frac{1+\nu}{1-\nu}\frac{\eps}{M_\nu}\right]^{\frac{1-\nu}{1+\nu}}\]

Denote $\gamma=\frac{1+\nu}{1+3\nu}\geqslant \frac{1}{2}.$ We have \begin{align*}
    (A_{k+1}^{1-\gamma}+A_k^{1-\gamma})(A_{k+1}^\gamma-A_k^\gamma)&=A_{k+1}-A_k+A_{k+1}^{\gamma}A_{k}^{1-\gamma}-A_k^\gamma A_{k+1}^{1-\gamma}=\\
    &=A_{k+1}-A_k+(A_{k+1}A_k)^{1-\gamma}(A_{k+1}^{2\gamma-1}-A_k^{2\gamma-1})\geqslant A_{k+1}-A_k.
\end{align*}Since $A_{k+1}=A_k+a_{k+1},$ 

\begin{equation}A^\gamma_{k+1}-A^\gamma_k\geqslant \frac{A_{k+1}-A_{k}}{A_{k+1}^{1-\gamma}+A_k^{1-\gamma}}\geqslant\frac{a_{k+1}}{2A_{k+1}^{1-\gamma}}\geqslant\frac{1}{2M_\nu^\frac{2}{1+3\nu}}\left[\frac{1+\nu}{1-\nu}\eps\right]^\frac{1-\nu}{1+3\nu} \label{A_k_universal}
\end{equation} Now we take a telescopic sum for $k=0,\ldots,N-1$ and get

\begin{equation}
    A_N-A_0=A_N\geqslant\left[\frac{1+\nu}{1-\nu}\right]^\frac{1-\nu}{1+\nu}\frac{N^\frac{1+3\nu}{1+\nu}\eps^\frac{1-\nu}{1+\nu}}{2^\frac{1+3\nu}{1+\nu}M_\nu^\frac{2}{1+\nu}}.
\end{equation}
To get the statement of the theorem it remains to notice that the algorithm is independent of the level of smoothness $\nu$. Hence, if the objective has H\"{o}lder continuous gradient for multiple $\nu\in[0,1]$, then $A_k$ will grow according to the greatest lower bound. Thus, we have 
\[A_k\geqslant\sup_{\nu\in[0,1]} \left[\frac{1+\nu}{1-\nu}\right]^\frac{1-\nu}{1+\nu}\frac{k^\frac{1+3\nu}{1+\nu}\eps^\frac{1-\nu}{1+\nu}}{2^\frac{1+3\nu}{1+\nu}M_\nu^\frac{2}{1+\nu}}. \]

\end{proof}
The convergence rate of the above algorithm is given by the following theorem.

\begin{theorem}\label{u-main-result}
If $f(x)$ is convex (or 1-weakly-quasi-convex) and has H\"{o}lder continuous (sub)gradients, method\\ UAGMsDR generates $x_k$ such that \[f(x_k)-f(x_*)\leqslant \frac{1}{A_k}V(x_*,x^0)+\frac{\eps}{2}\]
An $\eps$-accurate iterate $x_T$ is obtained in the number of iterations
\[N\leqslant\inf_{\nu\in[0,1]}2^{\frac{2+4\nu}{1+3\nu}}\left[\frac{1-\nu}{1+\nu}\right]^\frac{1-\nu}{1+3\nu}\left[\frac{M_\nu}{\eps}\right]^\frac{2}{1+3\nu}\Theta^\frac{1+\nu}{1+3\nu},\]
where $V(x_*,x^0)\leqslant \Theta$

\end{theorem}

\begin{proof}
According to the definition of $1$-weak-quasi-convexity 
$$l_k(x_*) = \sum_{i=0}^k a_{i+1}\{f(y^i) + \langle \nabla f(y^i), x_* - y^i \rangle\}\leqslant
$$
$$
\leqslant \sum_{i=0}^k a_{i+1}f(x_*)=A_{k+1}f(x_*).$$
By \eqref{eq:xkp1_ns} and \eqref{eq:beta_k_y_k_def_ns} we have $f(y^i)\leqslant f(x^i)\leqslant f(y^{i-1})\leqslant\ldots \leqslant f(x^0)$, so

 \[l_k(x_*) \leqslant \sum_{i=0}^k a_{i+1}\{(1 - \gamma)f(x^0) + \gamma f(x_*)\}.\]
From this inequality and \textbf{Theorem} \ref{universal-Ak-rate} we have
$$A_{k} f(x_{k}) \leqslant \min_{x \in \mathbb{R}^n} \psi_{k}(x) +\frac{A_k\eps}{2}\leqslant \psi_{k}(x_*)+\frac{A_k\eps}{2} = 
$$
$$
=l_{k-1}(x_*) + V(x_*,x^0)+\frac{A_k\eps}{2} \leqslant A_kf(x_*)+ V(x_*,x^0)+\frac{A_k\eps}{2}.$$
From here by rearranging the terms we obtain the first part of the statement of the theorem:
\[f(x_k)-f(x_*)\leqslant \frac{1}{A_k}V(x_*,x^0)+\frac{\eps}{2}.\] To get the seconds part we need to find $N$ such that the following bound holds:

\[\frac{1}{A_T}V(x_*,x^0)+\frac{\eps}{2}\leqslant\eps.\]

We have that $\forall \nu\in[0,1]$ (with possibly infinite $M_\nu$) \[A_T\geqslant \left[\frac{1+\nu}{1-\nu}\right]^\frac{1-\nu}{1+\nu}\frac{N^\frac{1+3\nu}{1+\nu}\eps^\frac{1-\nu}{1+\nu}}{2^\frac{1+3\nu}{1+\nu}M_\nu^\frac{2}{1+\nu}},\] so it is sufficient to guarantee \[\left[\frac{1+\nu}{1-\nu}\right]^\frac{1-\nu}{1+\nu}\frac{N^\frac{1+3\nu}{1+\nu}\eps^\frac{1-\nu}{1+\nu}}{2^\frac{1+3\nu}{1+\nu}M_\nu^\frac{2}{1+\nu}}\geqslant \frac{2V(x_*,x^0)}{\eps}, \] so after \[N\geqslant2^{\frac{2+4\nu}{1+3\nu}}\left[\frac{1-\nu}{1+\nu}\right]^\frac{1-\nu}{1+3\nu}\left[\frac{M_\nu}{\eps}\right]^\frac{2}{1+3\nu}\Theta^\frac{1+\nu}{1+3\nu}\] iterations accuracy $\eps$ is guaranteed. It remains to see that infimum over $\nu$ may be taken in this upper bound, since the method does not have $\nu$ as a parameter.
\end{proof}

This method also converges to a stationary point for non-convex objectives. 
\begin{theorem}
\label{Th:GradDecayNonConv_universal}
Let function $f$ be $L$-smooth and algorithm UAGMsDR be run for $N$ steps. Then
\[\min_{k=0,\ldots, N-1}\|\nabla f(y^k)\|^2_2\leqslant \frac{2L(f(x^0)-f(x_*))}{N}+L\eps.\]
\end{theorem}
\begin{proof}
We have that \begin{equation}
f(x^{k+1})\leqslant f(y^k)-\frac{1}{2L}\|\nabla f(y^k)\|^2_*+\frac{a_{k+1}\eps}{2A_{k+1}}\leqslant f(x^k)-\frac{1}{2L}\|\nabla f(y^k)\|^2_*+\frac{a_{k+1}\eps}{2A_{k+1}}. \label{grad_descent_guarantee_ns}
\end{equation}

Summing this up for $k=0,\ldots, N$, we obtain 

\[f(x^0)-f(x_*)\geqslant f(x^{0})-f(x^{N+1})\geqslant \frac{N}{2L}\min_{k=0,\ldots, N}\|\nabla f(y^k)\|^2_*-\sum_{k=0}^{N-1}\frac{a_{k+1}}{A_{k+1}}\frac{\eps}{2}\geqslant\frac{N}{2L}\min_{k=0,\ldots, N}\|\nabla f(y^k)\|^2_*-\frac{N\eps }{2}.\]
Consequently, we may guarantee \[\min_{k=0,\ldots, N-1}\|\nabla f(y^k)\|^2_2\leqslant \frac{2L(f(x^0)-f(x_*))}{N}+L\eps.\]
\end{proof}
\subsubsection{Online stopping criterion}

This method also has an efficient stopping criterion, provided the objective is convex.

By rewriting the statement of \textbf{Theorem~\ref{u-main-result}} we see that \[f(x^k)\leqslant\frac{\eps}{2}+\frac{1}{A_k}\psi_k(v^k)=\frac{\eps}{2}+\frac{1}{A_k}\min_{x\in\mathbb{R}^n}\left[\frac{1}{2}\|x^0-x\|^2_2+\sum_{i=0}^{k-1} a_{i+1}\left\{f(y^i) + \langle \nabla f(y^i), x - y^i \rangle\right\}\right]\]

As it was before, \[l^{k-1}(x)=\sum_{i=0}^{k-1} a_{i+1}\left\{f(y^i) + \langle \nabla f(y^i), x - y^i \rangle+\frac{\mu}{2}\|x-y^i\|^2_2\right\}\]Denote $R=\|x^0-x_*\|_2$   and \[\hat{f}^k=\min_{x:\ \|x-x_0\|\leqslant R} \frac{1}{A_k}l^{k-1}(x).\]

The constraint may be rewritten equivalently  as $\frac{1}{2}\|x-x_0\|^2\leqslant \frac{R^2}{2}$. By strong duality we see that \begin{align*}
\hat{f}^k&=\min_{x\in \mathbb{R}^n}\max_{\lambda\geqslant 0}\left\lbrace\frac{1}{A_k} l^{k-1}(x)+\lambda\left(\frac{1}{2}\|x_0-x\|^2-\frac{R^2}{2}\right)\right\rbrace\\&=\max_{\lambda\geqslant 0} \min_{x\in \mathbb{R}^n}\left\lbrace\frac{1}{A_k} l^{k-1}(x)+\lambda\left(\frac{1}{2}\|x_0-x\|^2-\frac{R^2}{2}\right)\right\rbrace.
\end{align*}

Now we set $\lambda=\frac{1}{A_k}$ and obtain \[\hat{f}^k\geqslant\frac{1}{A_k}\psi_k(v^k)-\frac{R^2}{2A_k}. \]

Then \[f(x^k)-\hat{f}^k\leqslant\frac{R^2}{2A_k}+\frac{\eps}{2}.\] But by convexity of $f(x)$ we have that $\forall k$ $\frac{1}{A_k}l^{k-1}(x)\leqslant f(x)$, which implies that $\hat{f}^k\leqslant f(x_*).$  Finally, we have

\[f(x^k)-f(x_*)\leqslant f(x^k)-\hat{f}^k\leqslant\frac{R^2}{2A_k}+\frac{\eps}{2},\] so the condition $f(x^k)-\hat{f}^k\leqslant \frac{\eps}{2}$ implies $f(x^k)-f(x_*)\leqslant \eps$ and is an efficient stopping criterion for the UAGMsDR method.

\subsection{Non-smooth strongly convex objectives}

It remains to combine the two ideas used previously into a method for the case of non-smooth strongly convex objective. However, while there exist functions which are globally both $L$-smooth and $\mu$-strongly convex, this is not the case for objectives with H\"{o}lder continuous gradients. Indeed, no function satisfies \[f(x)+\langle f(x),y-x\rangle+\frac{\mu}{2}\|y-x\|_2^{2}\leqslant f(y)\leqslant f(x)+\langle f(x),y-x\rangle+\frac{M_\nu}{1+\nu}\|y-x\|_2^{1+\nu}
\] for all $x,y\in\mathbb{R}^n$. However, we have already established that our method converges monotonously. Since strongly convex functions have bounded sublevel sets, we only need this pair of inequalities to hold true for $\forall x,y\in\mathcal{L}_f(f(x^0))=\{x|f(x)\leqslant f(x^0)\}.$ 

\begin{algorithm}[!h]
\caption{Universal Accelerated Gradient Method with Small-Dimensional Relaxation (UAGMsDR)}
\label{SCUAGMsDR}
\begin{algorithmic}[1]
\REQUIRE Accuracy $\eps$
\ENSURE $x^k$
\STATE Set $k = 0$, $A_0=0$, $x^0 = v^0$, $\psi_0(x) = \frac{1}{2}\|x_0-x_*\|^2_2$
\FOR{$k \geqslant 0$}
\STATE \begin{equation}
    \label{eq:beta_k_y_k_def_ns_sc}
    \beta_k = \arg\min_{\beta \in \left[0, 1 \right]} f\left(v^k + \beta (x^k - v^k)\right), \quad y^k = v^k + \beta_k (x^k - v^k).
\end{equation}
\STATE 

\begin{equation}
\label{eq:xkp1_opt_b_ns_sc}
h_{k+1} = \arg\min_{h \geqslant 0} f\left(y^k - h(\nabla f(y^k))^{\#}\right), \quad x^{k+1} = y^  k- h_{k+1}(\nabla f(y^k))^{\#}.
\end{equation}
Find $a_{k+1}$ from equation $f(y^k) - \frac{a_{k+1}^2}{2(A_{k} + a_{k+1})} \|\nabla f(y^k) \|_2^2+\frac{\mu\tau_k a_{k+1}}{2(\tau_{k}+\mu a_{k+1})(A_k+a_{k+1})}\|v^k-y^k\|_2^2 + \frac{\varepsilon a_{k+1}}{2(A_{k} + a_{k+1})} = f(x^{k+1})$.
\STATE  Set $A_{k+1} = A_{k} + a_{k+1}$.  
\STATE  Set $\psi_{k+1}(x) =  \psi_{k}(x) + a_{k+1}\{f(y^k) + \langle \nabla f(y^k), x - y^k \rangle+\frac{\mu}{2}\|x-y^k\|_2^2\}$.
\STATE $v^{k+1} = \arg\min_{x \in E} \psi_{k+1}(x)$
\STATE $k = k + 1$
\ENDFOR
\end{algorithmic}
\end{algorithm}
\begin{theorem}\label{universal-Ak-rate_sc}
For the algorithm UAGMsDR and $\mu$-strongly convex $f(x)$ with H\"{o}lder continuous (sub)gradients,
\begin{equation}\label{eq:main_recurrence_sc_ns}
    A_k f(x^k) \leqslant \min_{x \in \mathbb{R}^n} \psi_{k}(x) +\frac{A_{k}\eps}{2} = \psi_{k}(v^{k})+\frac{A_{k}\eps}{2}
\end{equation}
and
\begin{equation}\label{eq:A_k_max_rate}
    A_k\geqslant\max\left\lbrace \left[\frac{1+\nu}{1-\nu}\right]^\frac{1-\nu}{1+\nu}\frac{k^\frac{1+3\nu}{1+\nu}\eps^\frac{1-\nu}{1+\nu}}{2^\frac{1+3\nu}{1+\nu}M_\nu^\frac{2}{1+\nu}},\   \frac{1}{M_\nu}\left[\frac{1+\nu}{1-\nu}\frac{\eps}{M_\nu}\right]^{\frac{1-\nu}{1+\nu}}\left(1-M_\nu^{-\frac{1+\nu}{1+3\nu}}\left[\frac{1+\nu}{1-\nu}\frac{\eps}{M_\nu}\right]^{\frac{1-\nu}{1+3\nu}}\mu^\frac{1+\nu}{1+3\nu}\right)^{-(k-1)}\right\rbrace.
\end{equation}

Furthermore, an  $\eps$-accurate iterate $x_T$ is obtained in the number of iterations
\begin{equation}
    \label{eq:sc_ns_iterations}
    N\leqslant\inf_{\nu\in[0,1]}\min\left\lbrace2\left[\frac{1-\nu}{1+\nu}\right]^\frac{1-\nu}{1+3\nu}\left[\frac{M_\nu}{\eps}\right]^\frac{2}{1+3\nu}R^\frac{2+2\nu}{1+3\nu},
\frac{M_\nu^{\frac{2}{1+3\nu}}}{\eps^{\frac{1-\nu}{1+3\nu}}\mu^{\frac{1+\nu}{1+3\nu}}}\ln{\left(\left[\frac{1-\nu}{1+\nu}\right]^{\frac{1-\nu}{1+\nu}}\frac{M_\nu^{\frac{2}{1+\nu}}R^2}{\eps^{\frac{2}{1+\nu}}}\right)}\right\rbrace,
\end{equation}
where $\|x^0-x_*\|\leqslant R$.
\end{theorem}
\begin{proof}
Same as before, denote
$$l_k(x) = \sum_{i=0}^k a_{i+1}\{f(y^i) + \langle \nabla f(y^i), x - y^i \rangle +\frac{\mu}{2}\|x-y^k\|^2\}.$$

First, we prove inequality \eqref{eq:main_recurrence_sc_ns} by induction over $k$. For $k=0$, the inequality holds. Assume that 
$$A_{k}f(x^{k}) \leqslant \min_{x \in \mathbb{R}^n} \psi_{k}(x)+\frac{A_{k}\eps}{2} = \psi_{k}(v^{k}) +\frac{A_{k}\eps}{2}.$$
Then 
$$\psi_{k+1}(v^{k+1}) = \min_{x \in \mathbb{R}^n} \left\{ \psi_{k}(x) + a_{k+1}\{f(y^k) + \langle \nabla f(y^k), x - y^k \rangle+\frac{\mu}{2}\|x-y^k\|^2\} \right\}  $$
$$  \geqslant \min_{x \in \mathbb{R}^n} \left\{ \psi_{k}(v^k) +\frac{\tau_k}{2}\|x-v^k\|^2 + a_{k+1}\{f(y^k) + \langle \nabla f(y^k), x - y^k \rangle+\frac{\mu}{2}\|x-y^k\|^2\} \right\} $$
$$\geqslant \min_{x \in \mathbb{R}^n} \left\{ A_{k}f(x^k)-\frac{A_{k}\eps}{2} +\frac{\tau_k}{2}\|x-v^k\|^2 + a_{k+1}\{f(y^k) + \langle \nabla f(y^k), x - y^k \rangle+\frac{\mu}{2}\|x-y^k\|^2\} \right\}.$$
Here we used that $\psi_{k}$ is a $\tau_k$-strongly convex function with minimum at $v^k$.

By the definition of $\beta_k$ and $y_k$ in \eqref{eq:beta_k_y_k_def_sc}, we have 
$f(y^k) \leqslant f(x^k)$. By the optimality conditions in \eqref{eq:beta_k_y_k_def_sc}, there exists such a subgradient $\nabla f(y^k)$ that either
\begin{enumerate}
    \item $\beta_k = 0$, $\langle \nabla f(y^k),x^k - v^k \rangle \geqslant 0$, $y^k = v^k$;
    \item $\beta_k \in (0,1)$ and $\langle \nabla f(y^k),x^k - v^k \rangle = 0$, $y^k = v^k + \beta_k (x^k - v^k)$;
    \item $\beta_k = 1$ and $\langle \nabla f(y^k),x^k - v^k \rangle \leqslant 0$, $y^k = x^k$ .
\end{enumerate}
In all three cases,  $\langle \nabla f(y^k), v^k - y^k \rangle \geqslant 0$. Thus,
$$\psi_{k+1}(v^{k+1}) \geqslant \min_{x \in \mathbb{R}^n} \left\{ A_{k}f(y^k)-\frac{A_{k}\eps}{2} +\frac{\tau_k}{2}\|x-v^k\|^2 + a_{k+1}\{f(y^k) + \langle \nabla f(y^k), x - y^k \rangle+\frac{\mu}{2}\|x-y^k\|^2\} \right\}. $$

The explicit solution to this quadratic minimization problem is \[x=\frac{1}{\tau_{k+1}}(\tau_k v^k+\mu a_{k+1}y^k-a_{k+1}\nabla f(y^k)).\] By plugging in the solution and using $\langle \nabla f(y^k), v^k - y^k \rangle \geqslant 0$, we obtain 

\[\psi_{k+1}(v^{k+1})\geqslant  A_{k+1}f(y^k)-\frac{A_{k}\eps}{2} - \frac{a_{k+1}^2}{2\tau_{k+1}}\|\nabla f(y^k)\|_2^2+\frac{\mu\tau_k a_{k+1}}{2\tau_{k+1}}\|v^k-y^k\|^2.\]
Our next goal is to show that
\begin{equation} 
\label{eq:ind_fin_step_sc_ns}
A_{k+1}f(y^k)-\frac{A_{k}\eps}{2} - \frac{a_{k+1}^2}{2\tau_{k+1}}\|\nabla f(y^k)\|_2^2+\frac{\mu\tau_k a_{k+1}}{2\tau_{k+1}}\|v^k-y^k\|^2 \geqslant A_{k+1}f(x^{k+1})-\frac{A_{k+1}\eps}{2},
\end{equation}
which proves the induction step. But this is guaranteed by the choice of $a_{k+1}$ in the step 4 of the method as the solution of the equation\begin{equation}
    \label{eq:Th:Main_sc_ns_proof_1}
    f(y^k) - \frac{a_{k+1}^2}{2A_{k+1}\tau_{k+1}} \|\nabla f(y^k) \|_2^2+\frac{\mu\tau_k a_{k+1}}{2A_{k+1}\tau_{k+1}}\|v^k-y^k\|^2=f(x^{k+1})-\frac{a_{k+1}\eps}{2A_{k+1}}.
\end{equation}
It remains to show that this equation has a solution $a_{k+1} > 0$. Again, this is a quadratic equation with the greatest solution given by \[a_{k+1}=\frac{-S_{k,\eps}+\sqrt{S_{k,\eps}^2-8A_k\delta_k\tau_k(2\delta_{k,\eps}\mu+\|\nabla f(y_k)\|^2)}}{4\delta_{k,\eps}\mu+2\|\nabla f(y_k)\|_2^2},\] where $\delta_k=f(x^{k+1})-f(y^k)$, $\delta_{k,\eps}=\delta_k-\frac{\eps}{2}, S_{k,\eps}=2\delta_{k,\eps}\tau_k-2\mu A_k\delta_k+\mu\tau_k\|v^k-y^k\|^2$ 

Note that if the objective is $\mu$-strongly convex, it is true that $\forall x\in\mathbb{R}^n$ $f(x)-f(x_\ast)\leqslant \frac{1}{2\mu}\|\nabla f(x)\|^2$. Hence, \[2\delta_{k,\eps}\mu+\|\nabla f(y_k)\|^2\geqslant-\frac{\eps}{2}+f(x^{k+1})-f(x_*).\] This means that $a_{k+1}$ may only be negative or undefined if $f(x^{k+1})-f(x_*)\leqslant \frac{\eps}{2}$, which means that $f(x^{k+1})$ is already an  $\frac{\eps}{2}$-accurate solution to the problem. 

Let us estimate the rate of the growth for $A_k$. By the same argument as the one used in the proof of \textbf{Theorem~}\ref{universal-Ak-rate} we have \[\frac{a^2_k}{A_k\tau_k}\geqslant\frac{1}{M_\nu}\left[\frac{1+\nu}{1-\nu}\frac{\eps}{M_\nu}\right]^{\frac{1-\nu}{1+\nu}}\left[\frac{a_k}{A_k}\right]^{\frac{1-\nu}{1+\nu}}\] Using that we obtain

\[\frac{a^{\frac{1+3\nu}{1+\nu}}_k}{A_k^{\frac{2\nu}{1+\nu}}}\geqslant\frac{1}{M_\nu}\left[\frac{1+\nu}{1-\nu}\frac{\eps}{M_\nu}\right]^{\frac{1-\nu}{1+\nu}}(1+\mu A_k)\geqslant\frac{1}{M_\nu}\left[\frac{1+\nu}{1-\nu}\frac{\eps}{M_\nu}\right]^{\frac{1-\nu}{1+\nu}}\mu A_k,\] or, if we denote $\gamma=\frac{1+\nu}{1+3\nu}$, \[a_k\geqslant \frac{1}{M^\gamma_\nu}\left[\frac{1+\nu}{1-\nu}\frac{\eps}{M_\nu}\right]^{\frac{1-\nu}{1+3\nu}}\mu^\gamma A_k.\]

To get the left term in the stated lower bound on $A_k$ we write
\begin{equation}A^\gamma_{i}-A^\gamma_{i-1}\geqslant \frac{A_{i}-A_{i-1}}{A_{i}^{1-\gamma}+A_{i-1}^{1-\gamma}}\geqslant\frac{a_{i}}{2A_{i}^{1-\gamma}}\geqslant\frac{1}{2M_\nu^\frac{2}{1+3\nu}}\left[\frac{1+\nu}{1-\nu}\eps\right]^\frac{1-\nu}{1+3\nu}\left(1+\mu A_i \right)^\gamma. \label{A_k_sc_ns}
\end{equation} From this we obtain a weaker inequality \[A^\gamma_{i}-A^\gamma_i\geqslant\frac{1}{2M_\nu^\frac{2}{1+3\nu}}\left[\frac{1+\nu}{1-\nu}\eps\right]^\frac{1-\nu}{1+3\nu}.\] Now we telescope it for $i=0,\ldots,k$ and get

\begin{equation}
    A_k-A_0=A_k\geqslant\left[\frac{1+\nu}{1-\nu}\right]^\frac{1-\nu}{1+\nu}\frac{k^\frac{1+3\nu}{1+\nu}\eps^\frac{1-\nu}{1+\nu}}{2^\frac{1+3\nu}{1+\nu}M_\nu^\frac{2}{1+\nu}}.\label{A_k_sc_ns_2}
\end{equation}

To get the right term we observe that \[A_{k+1}=A_k+a_{k+1}\geqslant A_k+\frac{1}{M^\gamma_\nu}\left[\frac{1+\nu}{1-\nu}\frac{\eps}{M_\nu}\right]^{\frac{1-\nu}{1+3\nu}}\mu^\gamma A_{k+1},\] which leads to \[A_{k+1}\geqslant\left(1-\frac{1}{M^\gamma_\nu}\left[\frac{1+\nu}{1-\nu}\frac{\eps}{M_\nu}\right]^{\frac{1-\nu}{1+3\nu}}\mu^\gamma\right)^{-1}A_k.\] To use this bound we only need to estimate $A_1$, which we can do as follows: \[A_1=\frac{a_1^2}{A_1}\geqslant\frac{a_{1}^2}{(1+\mu A_1)A_{1}}=\frac{a_{1}^2}{\tau_{1}A_{1}} \geqslant \frac{1}{M_\nu}\left[\frac{1+\nu}{1-\nu}\frac{\eps}{M_\nu}\right]^{\frac{1-\nu}{1+\nu}}\]

By recursively applying the last bound we reach the desired result:

\[A_k\geqslant\frac{1}{M_\nu}\left[\frac{1+\nu}{1-\nu}\frac{\eps}{M_\nu}\right]^{\frac{1-\nu}{1+\nu}}\left(1-\frac{1}{M^\gamma_\nu}\left[\frac{1+\nu}{1-\nu}\frac{\eps}{M_\nu}\right]^{\frac{1-\nu}{1+3\nu}}\mu^\gamma\right)^{-(k-1)}\]

By combining both bounds we get the statement of the theorem:

\[A_k\geqslant\max\left\lbrace \left[\frac{1+\nu}{1-\nu}\right]^\frac{1-\nu}{1+\nu}\frac{k^\frac{1+3\nu}{1+\nu}\eps^\frac{1-\nu}{1+\nu}}{2^\frac{1+3\nu}{1+\nu}M_\nu^\frac{2}{1+\nu}},\   \frac{1}{M_\nu}\left[\frac{1+\nu}{1-\nu}\frac{\eps}{M_\nu}\right]^{\frac{1-\nu}{1+\nu}}\left(1-M_\nu^{-\frac{1+\nu}{1+3\nu}}\left[\frac{1+\nu}{1-\nu}\frac{\eps}{M_\nu}\right]^{\frac{1-\nu}{1+3\nu}}\mu^\frac{1+\nu}{1+3\nu}\right)^{-(k-1)}\right\rbrace\]

It has already been established in \textbf{Theorem~}\ref{SC_conv} that 
$$\min_{x \in \mathbb{R}^n} \psi_{k}(x) \leqslant \psi_{k}(x_*) = l_{k-1}(x_*) + \frac{1}{2}\|x_0-x_*\|^2_2 \leqslant A_k f(x_*) + \frac{1}{2}\|x_0-x_*\|^2_2.$$ In conjuction with \eqref{eq:main_recurrence_sc_ns} this gives us

\[f(x^N)-f(x_*)\leqslant\frac{R^2}{2A_T}+\frac{\eps}{2}.\]

The first term  in \eqref{eq:sc_ns_iterations}  has already been established previously. It remains to analyze rate of convergence if the maximum in \eqref{eq:A_k_max_rate} is achieved on the second ter,. We need to satisfy

\[\frac{R^2}{2A_T}+\frac{\eps}{2}\leqslant\eps.\]

\[M_\nu R^2 \left[\frac{1-\nu}{1+\nu}\frac{M_\nu}{\eps}\right]^{\frac{1-\nu}{1+\nu}}\left(1-M_\nu^{-\frac{1+\nu}{1+3\nu}}\mu^{\frac{1+\nu}{1+3\nu}}\left[\frac{1+\nu}{1-\nu}\frac{\eps}{M_\nu}\right]^\frac{1-\nu}{1+3\nu}\right)^{N+1}\leqslant\eps\]Taking the natural logarithm of both sides, we obtain

\[(N+1)\ln{\left(1-M_\nu^{-\frac{1+\nu}{1+3\nu}}\mu^{\frac{1+\nu}{1+3\nu}}\left[\frac{1+\nu}{1-\nu}\frac{\eps}{M_\nu}\right]^\frac{1-\nu}{1+3\nu}\right)}\leqslant\ln{\left(\frac{\eps^{\frac{2}{1+\nu}}}{M_\nu^{\frac{2}{1+\nu}}R^2}\left[\frac{1+\nu}{1-\nu}\right]^\frac{1-\nu}{1+\nu}\right)}.\]

We now use $\ln(1-x)\leqslant -x$ and $N<N+1$ to get the final result: with such $A_k$ $f(x^N)-f(x_*)\leqslant\eps$ if

\[N\geqslant\frac{M_\nu^{\frac{2}{1+3\nu}}}{\eps^{\frac{1-\nu}{1+3\nu}\mu^\frac{1+\nu}{1+3\nu}}}\ln{\left(\frac{\eps^{\frac{2}{1+\nu}}}{M_\nu^{\frac{2}{1+\nu}}R^2}\left[\frac{1+\nu}{1-\nu}\right]^\frac{1-\nu}{1+\nu}\right)}.\]

\end{proof}

Note that if $\nu=0$ we have 
\[ N\leqslant\frac{M_0^2}{\eps\mu}\ln{\left(\frac{M_0^2R^2}{\eps^2}\right)},\]which is optimal up to a logarithmic factor \cite{nemirovsky1983problem}.

\section{Application to problems with linear constraints}
\label{S:p-d}






In this section, we consider a minimization problem with linear equality. The idea is to construct the dual problem and solve it by our UAGMsDR method endowed with a step in the primal space. Following \cite{dvurechensky2017adaptive,dvurechensky2018computational}, we show that this modification solves simultaneously both primal and dual problems and allows to obtain convergence rate. 

Specifically, we consider the following minimization problem 
\begin{equation}
(P_1) \quad \quad \min_{x\in Q \subseteq E} \left\{ f(x) : \bm{A}x =b \right\},
\notag
\end{equation}
where $E$ is a finite-dimensional real vector space, $Q$ is a simple closed convex set, $\bm{A}$ is given linear operator from $E$ to some finite-dimensional real vector space $H$, $b \in H$ is given.
The Lagrange dual problem to Problem $(P_1)$ is
\begin{equation}
(D_1) \quad \quad \max_{\lambda \in \Lambda} \left\{ - \la \lambda, b \ra  + \min_{x\in Q} \left( f(x) + \la \bm{A}^T \lambda  ,x \ra \right) \right\}.
\notag
\end{equation}
Here we denote $\Lambda=H^*$.
It is convenient to rewrite Problem $(D_1)$ in the equivalent form of a minimization problem
\begin{align}
& (P_2) \quad \min_{\lambda \in \Lambda} \left\{   \la \lambda, b \ra  + \max_{x\in Q} \left( -f(x) - \la \bm{A}^T \lambda  ,x \ra \right) \right\}. \notag
\end{align}
We denote
\begin{equation}
\vp(\lambda) =  \la \lambda, b \ra  + \max_{x\in Q} \left( -f(x) - \la \bm{A}^T \lambda  ,x \ra \right).
\label{eq:vp_def}
\end{equation}
Since $f$ is convex, $\vp(\lambda)$ is a convex function and, by Danskin's theorem, its subgradient is equal to (see e.g. \cite{nesterov2005smooth})
\begin{equation}
\nabla \vp(\lambda) = b - \bm{A} x (\lambda)
\label{eq:nvp}
\end{equation}
where $x (\lambda)$ is some solution of the convex problem
\begin{equation}
\max_{x\in Q} \left( -f(x) - \la \bm{A}^T \lambda  ,x \ra \right).
\label{eq:inner}
\end{equation}

In what follows, we make the following assumptions about the dual problem $(D_1)$
\begin{itemize}
    \item Subgradient of the objective function $\vp(\lambda)$ satisfies H\"older condition \eqref{eq:hold_cond} with constant $M_{\nu}$.
    \item The dual problem $(D_1)$ has a solution $\lambda^*$ and there exist some $R>0$ such that
	\begin{equation}
	\|\lambda^{*}\|_{2} \leqslant R < +\infty. 
	\label{eq:l_bound}
	\end{equation}
\end{itemize}
It is worth noting that the quantity $R$ will be used only in the convergence analysis, but not in the algorithm itself. As it was pointed in \cite{yurtsever2015universal}, the first assumption is reasonable. Namely, if the set $Q$ is bounded, then $\nabla \vp(\lambda)$ is bounded and \eqref{L} holds with $\nu = 0$. If $f(x)$ is uniformly convex, i.e., for all $x,y \in Q$, $\la \nabla f(x) - \nabla f(y) \ra \geqslant \mu \|x-y\|^{\rho}$, for some $\mu > 0$, $\rho \geqslant 2$, then $\nabla \vp(\lambda)$ satisfies \eqref{L} with $\nu = \frac{1}{\rho-1}$, $M_{\nu} = \left(\frac{\|\bm{A}\|_{E \to H}^2}{\mu}\right)^{\frac{1}{\rho-1}}$. Here the norm of an operator $\bm{A}:E_1 \to E_2$ is defined as follows
$$
\|\bm{A}\|_{E_1 \to E_2} = \max_{x \in E_1,u \in E_2^*} \{\la u, \bm{A} x \ra : \|x\|_{E_1} = 1, \|u\|_{E_2,*} = 1 \}.
$$

We choose Euclidean proximal setup, which means that we introduce Euclidean norm $\|\cdot \|_2$ in the space of vectors $\lambda$ and choose the prox-function $d(\lambda) = \frac12\|\lambda\|_2^2$. Then, we have $V[\zeta](\lambda) = \frac12\|\lambda-\zeta\|_2^2$.
Our primal-dual algorithm for Problem $(P_1)$ is listed below as Algorithm \ref{Alg:PDULSGD}. 

\begin{algorithm}[h!]
\caption{PDUGDsDR}
\label{Alg:PDULSGD}
{\small
\begin{algorithmic}[1]
   \REQUIRE starting point $\lambda_0 = 0$, accuracy $\tilde{\eps}_f,\tilde{\eps}_{eq} > 0$.
   \STATE Set $k=0$, $A_0=\alpha_0=0$, $\eta_0=\zeta_0=\lambda_0=0$.
   \REPEAT
		\STATE $\beta_k = \arg\min_{\beta \in \left[0, 1 \right]} \vp\left(\zeta^k + \beta (\eta^k - \zeta^k)\right)$; $\lambda^k = \zeta^k + \beta_k (\eta^k - \zeta^k)$ 
		\STATE $h_{k+1} = \arg\min_{h \geqslant 0} \vp\left(\lambda^k - h\nabla \vp(\lambda^k)\right)$; $\eta^{k+1} = \lambda^k- h_{k+1}\nabla \vp(\lambda^k)$ // Choose $\nabla \vp(\lambda^k)$ : $\langle \nabla \vp(\lambda^k), \zeta^k - \lambda^k \rangle \geqslant 0$
		\STATE Choose $a_{k+1}$ from $\vp(\eta^{k+1}) = \vp(\lambda^k) - \frac{a_{k+1}^2}{2A_{k+1}} \|\nabla \vp(\lambda^k) \|_2^2 + \frac{\varepsilon a_{k+1}}{2A_{k+1}}$ // $A_{k+1} = A_{k} + a_{k+1}$
		\STATE $\zeta^{k+1} = \zeta^k - a_{k+1}\nabla \vp(\lambda^k)$
		\STATE Set
				\begin{equation}
					\hat{x}^{k+1} = \frac{1}{A_{k+1}}\sum_{i=0}^{k} a_{i+1} x(\lambda^i) = \frac{a_{k+1}x(\lambda^{k})+A_k\hat{x}^{k}}{A_{k+1}}.
				\notag
				\end{equation}
			\STATE Set $k=k+1$.
  \UNTIL{$|f(\hat{x}^{k+1})+\vp(\eta^{k+1})| \leqslant \tilde{\eps}_f$, $\|\bm{A}\hat{x}^{k+1}-b\|_{2} \leqslant \tilde{\eps}_{eq}$.}
	\ENSURE The points $\hat{x}^{k+1}$, $\eta^{k+1}$.	
\end{algorithmic}
}
\end{algorithm}

\begin{theorem}
\label{Th:PD_rate}
Let the objective $\vp$ in the problem $(P_2)$ have H\"older-continuous subgradient and the solution of this problem be bounded, i.e. $\|\lambda^*\|_2 \leqslant R$. Then, for the sequence $\hat{x}^{k+1},\eta^{k+1}$, $k\geqslant 0$, generated by Algorithm \ref{Alg:PDULSGD}, 
\begin{align}
&\|\bm{A} \hat{x}^k - b \|_2 \leqslant \frac{2R}{A_k}+ \frac{\eps}{2R}, \quad |\vp(\eta^k) + f(\hat{x}^k)| \leqslant \frac{2R^2}{A_k}+ \frac{\eps}{2},
\label{eq:untileq}
\end{align}
where $A_k \geqslant \left[\frac{1+\nu}{1-\nu}\right]^\frac{1-\nu}{1+\nu}\frac{k^\frac{1+3\nu}{1+\nu}\eps^\frac{1-\nu}{1+\nu}}{2^\frac{1+3\nu}{1+\nu}M_\nu^\frac{2}{1+\nu}}$.
\label{Th:PDASTMConv}
\end{theorem}

\begin{proof}
The proof mostly follows the steps of our previous work \cite{chernov2016fast}, but we give the proof for the reader's convenience. The main difference is that here we use universal method. From Theorem \ref{Main}, since $\zeta_0=0$, we have, for all $k\geqslant 0$,  
\begin{equation}
A_k \vp(\eta^k) \leqslant \min_{\lambda \in \Lambda} \left\{  \sum_{i=0}^{k-1}{a_{i+1} \left( \vp(\lambda^i) + \la \nabla \vp(\lambda^i), \lambda-\lambda^i \ra \right) } + \frac{1}{2} \|\lambda\|_2^2 \right\} + \frac{A_k\eps}{2}
\label{eq:FGM_compl}
\end{equation}
Let us introduce a set $\Lambda_R =\{\lambda:  \|\lambda\|_2 \leqslant 2R \}$ where $R$ is given in \eqref{eq:l_bound}. Then, from \eqref{eq:FGM_compl}, we obtain
\begin{align}
A_k \vp(\eta^k) &\leqslant \min_{\lambda \in \Lambda} \left\{  \sum_{i=0}^{k-1}{a_{i+1} \left( \vp(\lambda^i) + \la \nabla \vp(\lambda^i), \lambda-\lambda^i \ra \right) } + \frac{1}{2} \|\lambda\|_2^2 \right\} + \frac{A_k\eps}{2} \notag \\
&\leqslant \min_{\lambda \in \Lambda_R} \left\{  \sum_{i=0}^{k-1}{a_{i+1} \left( \vp(\lambda^i) + \la \nabla \vp(\lambda^i), \lambda-\lambda^i \ra \right) } + \frac{1}{2} \|\lambda\|_2^2 \right\} + \frac{A_k\eps}{2} \notag \\
&\leqslant \min_{\lambda \in \Lambda_R} \left\{  \sum_{i=0}^{k-1}{a_{i+1} \left( \vp(\lambda^i) + \la \nabla \vp(\lambda^i), \lambda-\lambda^i \ra \right) } \right\} + 2R^2 + \frac{A_k\eps}{2}.
\label{eq:proof_st_1}
\end{align}
On the other hand, from the definition \eqref{eq:vp_def} of $\vp(\lambda)$, we have
\begin{align}
 \vp(\lambda^i) &  = \la \lambda^i, b \ra + \max_{x\in Q} \left( -f(x) - \la \bm{A}^T \lambda^i ,x \ra \right) \notag \\
& = \la \lambda^i, b \ra  - f(x(\lambda^i)) - \la \bm{A}^T \lambda^i  ,x(\lambda^i) \ra . \notag
\end{align}
Combining this equality with \eqref{eq:nvp}, we obtain
\begin{align}
\vp(\lambda^i) - \la \nabla \vp (\lambda^i), \lambda^i \ra & = \la \lambda^i, b \ra - f(x(\lambda^i)) - \la \bm{A}^T \lambda^i  ,x(\lambda^i) \ra \notag \\
& \hspace{1em} - \la b-\bm{A} x(\lambda^i),\lambda^i \ra  = - f(x(\lambda^i)). \notag
\end{align}
Summing these equalities from $i=0$ to $i=k-1$ with the weights $\{a_{i+1}\}_{i=0,...k-1}$, we get, using the convexity of $f$
\begin{align}
  \sum_{i=0}^{k-1}{a_{i+1} \left( \vp(\lambda^i) + \la \nabla \vp(\lambda^i), \lambda-\lambda^i \ra \right) }   &= -\sum_{i=0}^{k-1} a_{i+1} f(x(\lambda^i)) + \sum_{i=0}^{k-1} a_{i+1} \la (b-\bm{A} x(\lambda^i), \lambda \ra  \notag \\
& \leqslant -A_k f(\hat{x}^k) + A_{k} \la b-\bm{A} \hat{x}^{k}, \lambda\ra . \notag
\end{align}
Substituting this inequality to \eqref{eq:proof_st_1}, we obtain
\begin{align}
A_k \vp(\eta^k)  \leqslant &-A_kf(\hat{x}^k)  + A_k \min_{\lambda \in \Lambda_R} \left\{  \la b-\bm{A} \hat{x}^k, \lambda\ra \right\} + 2R^2 + \frac{A_k\eps}{2}. \notag
\end{align}
Finally, since 
$\max_{\lambda \in \Lambda_R} \left\{  \la -b+\bm{A} \hat{x}^k, \lambda \ra \right\} = 2 R \|\bm{A} \hat{x}^k - b \|_2$,
we obtain
\begin{equation}
\vp(\eta^k) + f(\hat{x}^k) +2 R \|\bm{A} \hat{x}^k - b \|_2 \leqslant \frac{2R^2}{A_k} + \frac{\eps}{2}.
\label{eq:vpmfxh}
\end{equation}
Since  $\lambda^*$ is an optimal solution of Problem $(D_1)$, we have, for any $x \in Q$
$$
Opt[P_1]\leqslant f(x) + \la  \lambda^*, \bm{A} x -b \ra .
$$
Using the assumption \eqref{eq:l_bound}, we get
\begin{equation}
f(\hat{x}^k) \geqslant Opt[P_1]- R \|\bm{A} \hat{x}^k - b \|_2 .
\label{eq:fxhat_est}
\end{equation}
Hence,
\begin{align}
 \vp(\eta^k) + f(\hat{x}^k)  & = \vp(\eta^k) - Opt[P_2]+Opt[P_2] + Opt[P_1]  - Opt[P_1] + f(\hat{x}^k)  \notag \\
& =\vp(\eta^k)  - Opt[P_2]-Opt[D_1]+Opt[P_1]  - Opt[P_1] + f(\hat{x}^k)  
 \notag \\
&  \geqslant  - Opt[P_1] + f(\hat{x}^k) \stackrel{\eqref{eq:fxhat_est}}{\geqslant} - R \|\bm{A} \hat{x}^k - b \|_2.
\label{eq:aux1}
\end{align}
This and \eqref{eq:vpmfxh} give
\begin{equation}
R \|\bm{A} \hat{x}^k - b \|_2 \leqslant \frac{2R^2}{A_k} + \frac{\eps}{2}.
\label{eq:R_norm_est}
\end{equation}
Hence, we obtain
\begin{equation}
\vp(\eta^k) + f(\hat{x}^k) \stackrel{\eqref{eq:aux1},\eqref{eq:R_norm_est}}{\geqslant} - \frac{2R^2}{A_k}- \frac{\eps}{2}.
\label{eq:vppfxhatgeq}
\end{equation}
On the other hand, we have 
\begin{equation}
\vp(\eta^k) + f(\hat{x}^k) \stackrel{\eqref{eq:vpmfxh}}{\leqslant} \frac{2R^2}{A_k}+ \frac{\eps}{2}.
\label{eq:vppfxhatleq}
\end{equation}
Combining \eqref{eq:R_norm_est}, \eqref{eq:vppfxhatgeq}, \eqref{eq:vppfxhatleq}, we conclude
\begin{align}
&\|\bm{A} \hat{x}^k - b \|_2 \leqslant \frac{2R}{A_k}+ \frac{\eps}{2R}, \notag \\
&|\vp(\eta^k) + f(\hat{x}^k)| \leqslant \frac{2R^2}{A_k}+ \frac{\eps}{2}.
\end{align}
From Theorem \ref{universal-Ak-rate}, for any $k\geqslant 0$, $A_k \geqslant \left[\frac{1+\nu}{1-\nu}\right]^\frac{1-\nu}{1+\nu}\frac{k^\frac{1+3\nu}{1+\nu}\eps^\frac{1-\nu}{1+\nu}}{2^\frac{1+3\nu}{1+\nu}M_\nu^\frac{2}{1+\nu}}$. Combining this and \eqref{eq:untileq}, we obtain the statement of the Theorem.

\end{proof}
Let us make a remark on complexity. As it can be seen from Theorem \ref{Th:PD_rate}, whenever $A_k \geqslant 2R^2/\eps$, the error in the objective value and equality constraints is smaller than $\eps$. A the same time, using the lower bound for $A_k$, we obtain that the number of iterations to achieve this accuracy is $O\left(\left(\frac{M_\nu^\frac{2}{1+\nu}R^2}{\eps^\frac{2}{1+\nu}} \right)^{\frac{1+\nu}{1+3\nu}}\right)$.




\section{Numerical experiments}
\label{S:experiments}

\subsection{Smooth convex problem}
We tested the AGMsDR method on the problem of minimizing \[f(x)=\frac{L}{8}(x_1^2+\sum\limits_{i=1}^{n-1} (x_i-x_{i+1}^2)+x_n^2)-\frac{L}{4}x_1,\] where $L=10$ and the diminsionality $n=1000$. This is the function used to derive the lower complexity bound for the class of $L$-smooth convex objectives in \cite{nesterov2004introduction}. The results are presented below. The method was compared to the Linear Coupling method \cite{allen2014linear} and the Conjugate Gradients and BFGS methods implemented in the SciPy python package.

For all four methods the initial point was the origin. The results are presented below.
\begin{figure}[!h]%
    \centering
    \advance\leftskip-4cm
    \advance\rightskip-4cm
    \subfloat{{\includegraphics[width=8cm]{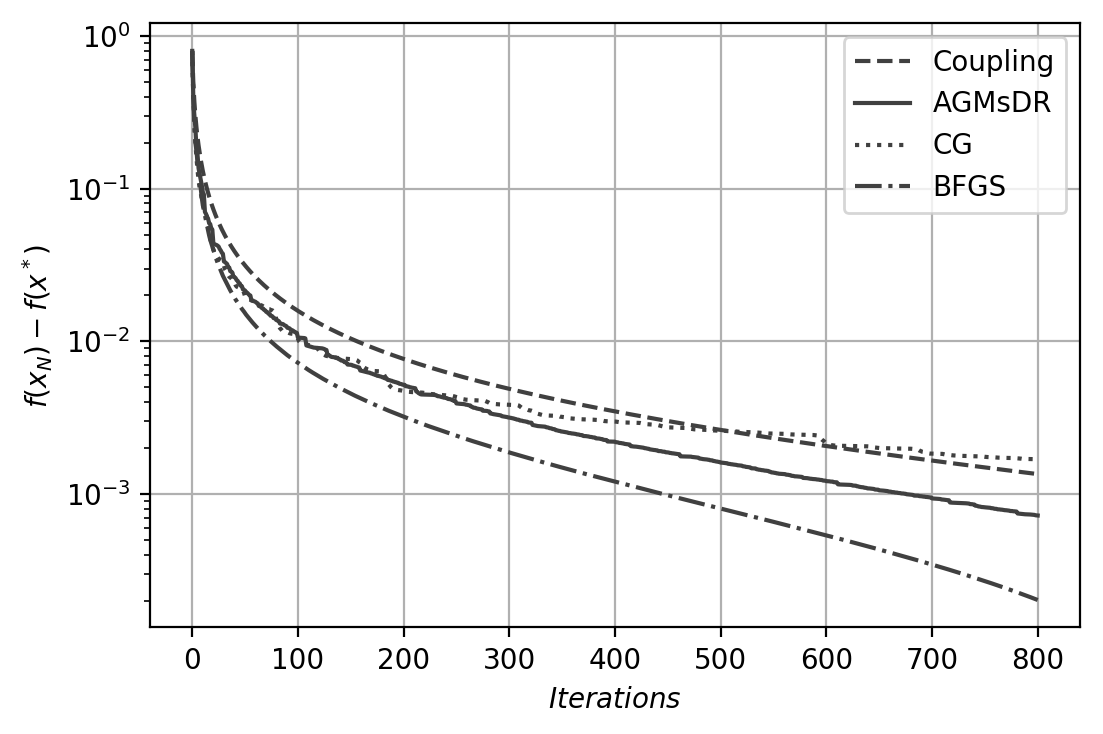} }}%
    \quad
    \subfloat{{\includegraphics[width=8cm]{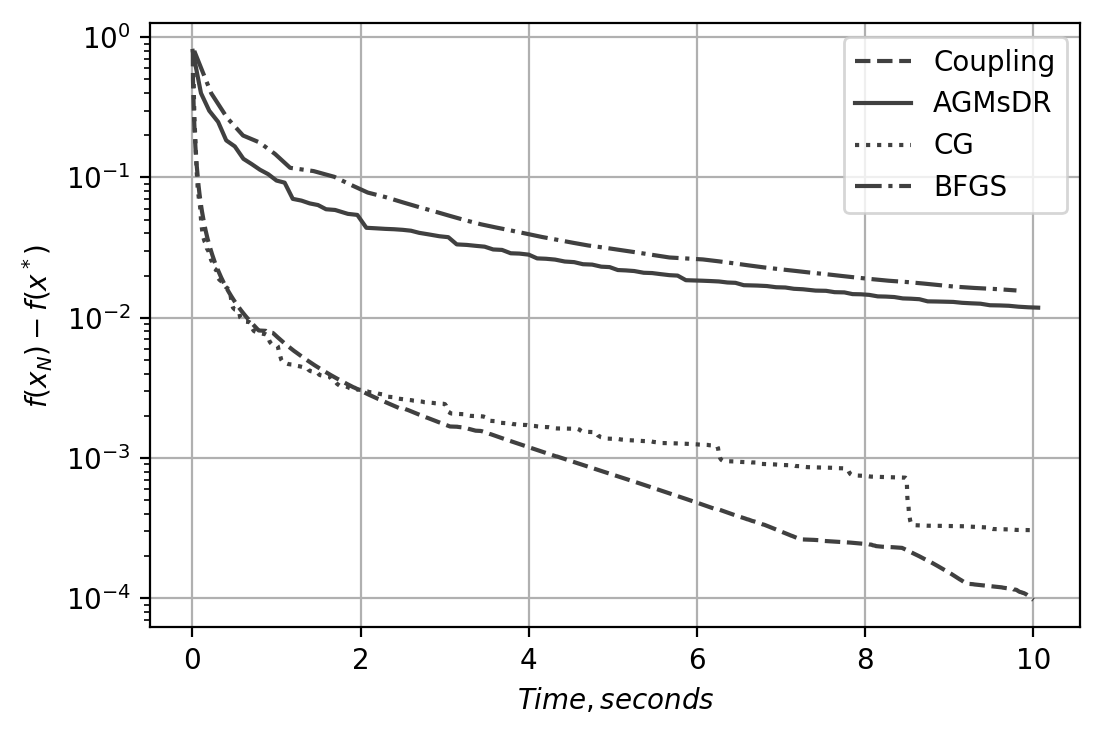} }}%
    \caption{Convergence for the smooth problem.}%
    \label{fig:smooth}%
\end{figure}

While utilizing line-search slightly improves the convergence rate in terms of required iterations compared to methods with fixed step sizes, the inherent complexity of line-search significantly increases the cost of each iteration. In the end, first-order methods with fixed step-sizes showed best results.  
\subsection{Non-smooth convex problem}
We compared three different universal methods -- UAGMsDR from this paper, Universal Linear Coupling Method from \cite{guminov2017universal} and Universal Fast Gradient Method from \cite{nesterov2015universal} -- on the MAXQ problem \cite{haarala2004new}:

\[f(x)=\max_{1\leqslant i\leqslant n} x_i^2=\|x\|_\infty^2, \] with the initial point

\[
    x^0_i = \left\{\begin{array}{rl}
        i,  &\text{for } i=1,\ldots,n/2\\
        -i, &\text{for } i=n/2+1,\ldots,n
        \end{array}\right.
  \] and $n=100$.
For all methods the accuracy $\eps$ was set to $5\cdot 10^{-4}$.

\begin{figure}[!h]
    \centering
    \includegraphics{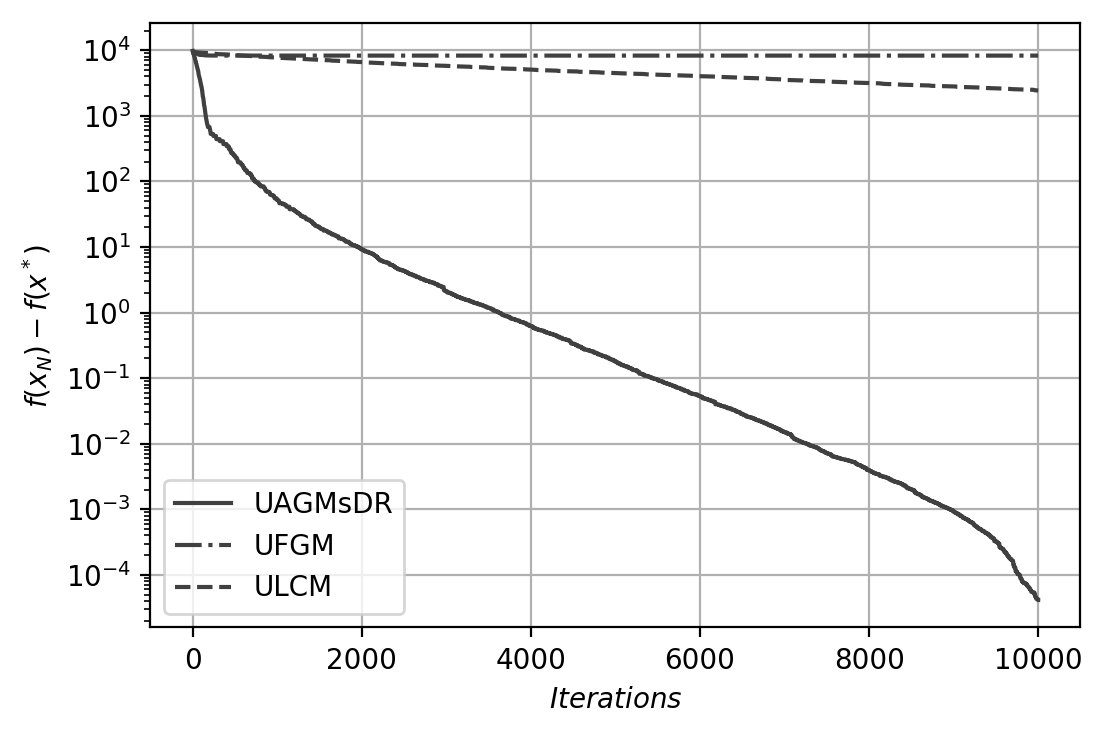}
    \caption{Convergence for the non-smooth problem.}
    \label{fig:nonsmooth}
\end{figure}

Even though all three methods have identical (up to a small constant multiplicative factor) theoretical convergence rates, for this problem the UAGMsDR method demonstrated practically linear convergence rate. It seems that using two line-searches in orthogonal directions helps the method use the fact that the graph of the function is, in a sense, similar to a quadratic function.

\subsection{Non-convex problem}
We consider the following non-convex objective with unique extremal point (which is the global minimum) at  $(1,1,1,...,1)$:
\[f(x) = \frac{1}{4}(x_1-1)^2 + \sum_{i=1}^n (x_{i+1} - 2x_{i}^2 + 1)^2\]
 and with the initial point $x^0=(-1,-1,-1,...,-1)$. Even with low dimensionality $n \simeq 15$ this problem is very hard. There are points $x$ such that $\|\nabla f(x)\|_2 \simeq 10^{-8}$, while at the same time typically $f(x) - f(x_*) = f(x) \simeq \frac{1}{2}$.
 
 To perform exact line-search for this problem we utilized the fact that the objective is a polynomial. For example, to minimize the objective over the line $\{y^k-h\nabla f(y^k)| h\in\mathbb{R}\}$ it is then sufficient to find the roots of the third degree polynomial $\frac{\partial}{\partial h} f(y^k-h\nabla f(y^k))$ and choose the one which corresponds to the lesser value of $f(y^k-h\nabla f(y^k))$. For all universal methods the accuracy $\eps$ was set to $5\cdot 10^{-4}$.

\begin{figure}[!h]%
    \centering
    \advance\leftskip-4cm
    \advance\rightskip-4cm
    \subfloat{{\includegraphics[width=8cm]{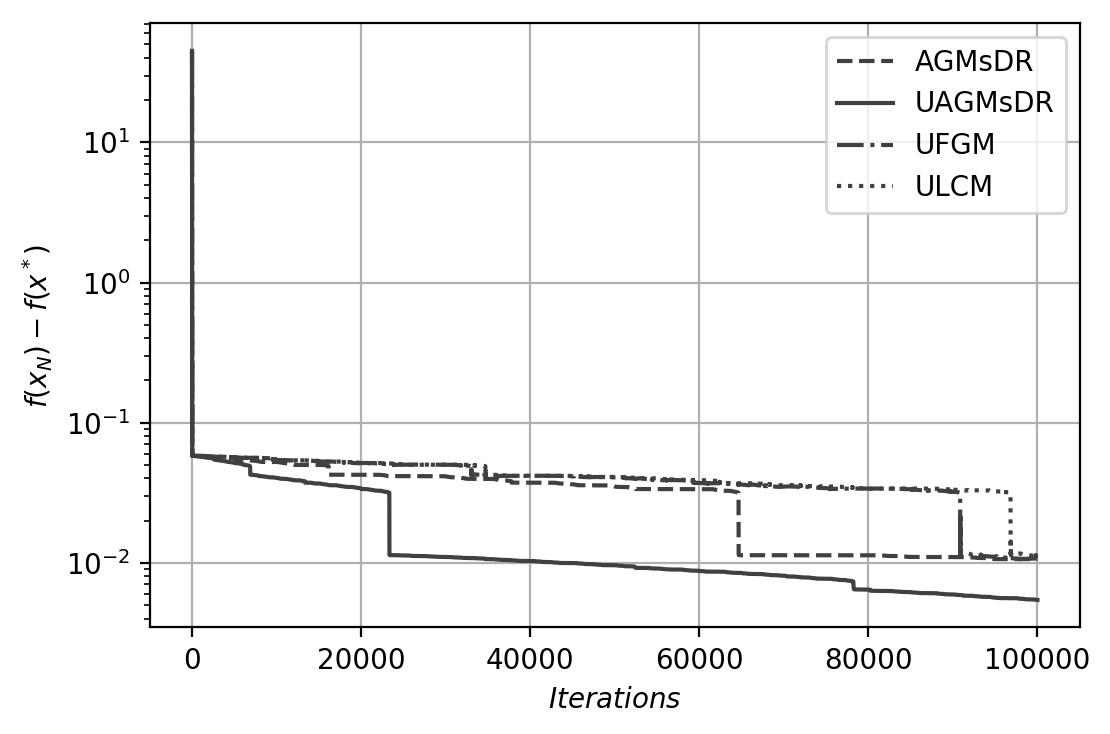} }}%
    \quad
    \subfloat{{\includegraphics[width=8cm]{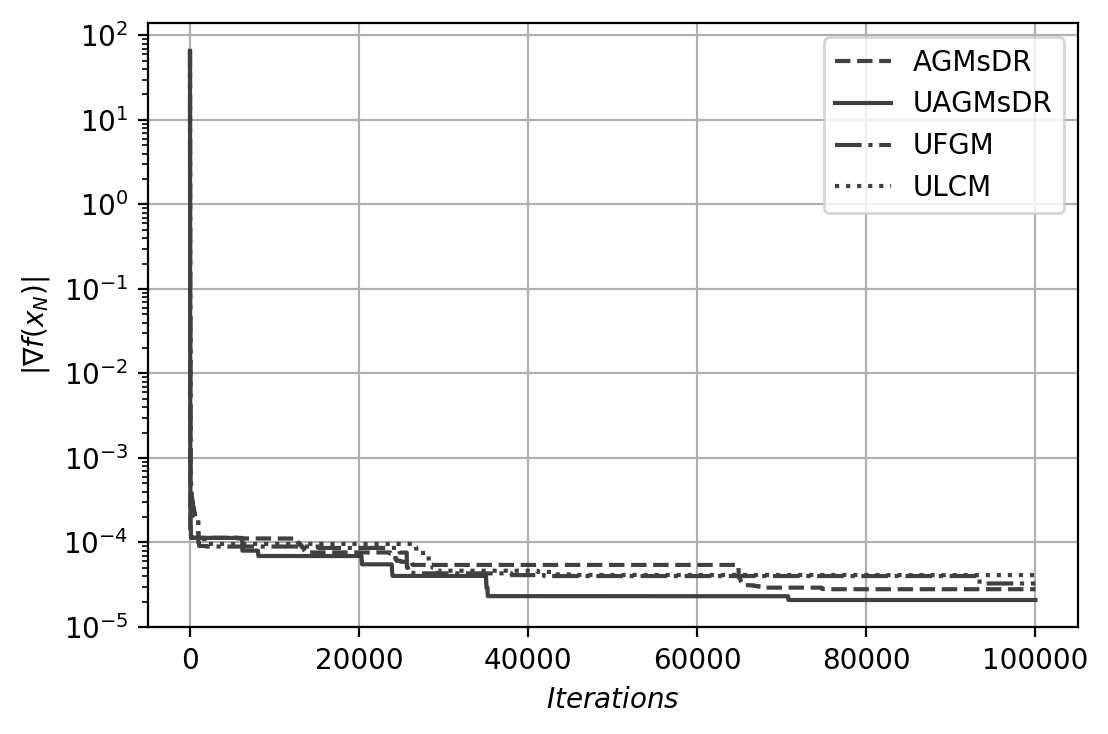} }}%
    \caption{Convergence for the smooth problem.}%
    \label{fig:nonconvex}%
\end{figure}

The universal UAGMsDR method demonstrates best performance, both in terms of convergence in gradient and the function's value.

\section{Funding}
The work in Section \ref{S:smooth} was funded by Russian Science Foundation (project 18-71-10108). The work in Section \ref{S:univ} was supported by grant RFBR 18-29-03071 mk. The work in Section \ref{S:p-d} was supported Grant of the President of the Russian Federation MD-1320.2018.1 and RFBR  18-31-20005 mol\_a\_ved. 

\bibliographystyle{tfs}
\bibliography{references,lib}

\end{document}